\documentclass[11pt]{amsart}

\usepackage{amsmath,amssymb,amsthm}
\usepackage{thmtools,enumerate}
\usepackage{mathtools}

\usepackage[german,english]{babel}
\usepackage[autostyle]{csquotes}

\usepackage[all]{xy}
\usepackage{pstricks}

\usepackage{tikz-cd}
\usetikzlibrary{babel}

\usepackage{tabularx}

\usepackage{hyperref}
\hypersetup{
	colorlinks=true,
	linkcolor=red,
	citecolor=blue}

\theoremstyle{plain}

\newtheorem{thm}{Theorem}[section]
\newtheorem{prop}[thm]{Proposition}
\newtheorem{lem}[thm]{Lemma}
\newtheorem{cor}[thm]{Corollary}

\newtheorem{thmA}{Theorem}

\newtheorem*{CC}{Cone Conjecture}
\newtheorem*{GenCC}{Generalised Cone Conjecture}

\theoremstyle{definition}

\newtheorem{dfn}[thm]{Definition}
\newtheorem{dfn-prop}[thm]{Definition-Proposition}
\newtheorem{rem}[thm]{Remark}

\newcommand{\Z}{\mathbb{Z}}
\newcommand{\N}{\mathbb{N}}
\newcommand{\Q}{\mathbb{Q}}
\newcommand{\R}{\mathbb{R}}
\newcommand{\C}{\mathbb{C}}
\newcommand{\pr}{\mathbb{P}}
\newcommand{\OO}{\mathcal{O}}
\newcommand{\cal}[1]{\mathcal{#1}}

\newcommand{\smperp}{{\scriptscriptstyle{\perp}}}
\newcommand{\smone}{{\scriptscriptstyle{(1)}}}
\newcommand{\smonestar}{{\scriptscriptstyle{(1)*}}}
\newcommand{\smonestarstar}{{\scriptscriptstyle{(1)**}}}

\DeclareMathOperator{\mult}{mult}
\DeclareMathOperator{\Exc}{Exc}
\DeclareMathOperator{\Pic}{Pic}
\DeclareMathOperator{\Eff}{Eff}

\DeclareMathOperator{\Nef}{Nef}
\DeclareMathOperator{\Amp}{Amp}
\DeclareMathOperator{\NefE}{Nef^e}
\DeclareMathOperator{\Comp}{Comp}
\DeclareMathOperator{\Cr}{Cr}
\DeclareMathOperator{\GL}{GL}
\DeclareMathOperator{\Aut}{Aut}
\DeclareMathOperator{\Int}{Int}
\DeclareMathOperator{\Supp}{Supp}

\DeclareMathOperator{\id}{Id}
\DeclareMathOperator{\nod}{nod}

\DeclareMathOperator{\Spec}{Spec}
\DeclareMathOperator{\sm}{sm}

\begin{document}
	
	\title[Generalised Cone Conjecture, I: Beyond Calabi--Yau]{Generalised Cone Conjecture, I:\\ Beyond Calabi--Yau}
 
	\author[V.\ Lazi\'c]{Vladimir Lazi\'c}
	\address{Fachrichtung Mathematik, Campus, Geb\"aude E2.4, Universit\"at des Saarlandes, 66123 Saarbr\"ucken, Germany}
	\email{lazic@math.uni-sb.de}

        \author[I.\ Stenger]{Isabel Stenger}
        \address{Institute of Algebraic Geometry, Leibniz University Hannover, Welfengarten 1, 30167 Hannover, Germany}
        \email{stenger@math.uni-hannover.de}
		
	\author[Z.\ Xie]{Zhixin Xie}
	\address{Institut \'Elie Cartan de Lorraine, Universit\'e de Lorraine, 54506 Nancy, France}
	\email{zhixin.xie@univ-lorraine.fr}
	
	
	\begin{abstract}
		This is the first of two papers in which we prove the Generalised Cone Conjecture on surfaces.
	\end{abstract}

	\maketitle

	\begingroup
		\hypersetup{linkcolor=black}
		\setcounter{tocdepth}{1}
		\tableofcontents
	\endgroup
	
\section{Introduction}

This is the first of two papers in which we prove the Generalised Cone Conjecture on surfaces. The second paper is \cite{LSX26}.

The Cone Conjecture has emerged in recent years as one of the most attractive open problems in birational geometry. In its classical form, it predicts that on a variety $X$ with mild singularities and with numerically trivial canonical class, the nef cone of $X$ -- which very often is itself not rational polyhedral -- is rational polyhedral up to the action of the automorphism group of $X$. This statement and its version involving the birational automorphism group have structural consequences in birational geometry: for instance, they imply that the number of minimal models of a smooth projective variety is finite up to abstract isomorphism. The Cone Conjecture is known in full generality only in dimension two.

It is known that this conjecture does not hold in its original form if the singularities of $X$ are bad enough. However, many pathological examples disappear if one works in the category of \emph{klt Calabi--Yau generalised pairs}, which we introduce below. This category includes, for instance, all smooth varieties with nef anticanonical class: these varieties are our main interest. We concentrate in this paper on generalised pairs of dimension two.

Our main results are:
\begin{enumerate}[\normalfont (i)]
	\item on a \emph{very general} klt Calabi--Yau generalised pair of dimension two there exists a group of geometric origin such that there \textbf{does exist} a rational polyhedral fundamental domain for the action of that group on the nef cone,
	\item there exists a klt Calabi--Yau generalised pair of dimension two such that for any group of geometric origin, there \textbf{does not exist} a rational polyhedral fundamental domain for the action of that group on the nef cone.
\end{enumerate}

We elaborate below in detail on why (i) is very surprising, and why -- as a consequence of (i) -- part (ii) is also very surprising. We explain the significance of our results in the context of broader literature and how they connect to several other works.

\medskip

\noindent{\sc Brief history of the Cone Conjecture}

\medskip

Before we continue, we fix some terminology. Let $V$ be a real vector space and let $\mathcal C\subseteq V$ be a convex cone. Let $G$ be a subgroup of $\GL(V)$ preserving $\mathcal C$. Then a \emph{fundamental domain} for the action of $G$ on $\mathcal C$ is a cone $\Pi\subseteq\mathcal C$ such that $\mathcal C=\bigcup_{g\in G}g(\Pi)$ and $\Int \Pi\cap \Int g(\Pi)=\emptyset$ unless $g$ is the identity.

If $X$ is a projective variety, the Cone Conjecture is a statement about the behaviour of the effective nef cone $\Nef^e(X) \coloneqq \Nef(X)\cap\Eff(X)$ under the action of the group $\Aut(X)$, and the behaviour of the effective movable cone $\operatorname{Mov}^e(X) \coloneqq \overline{\operatorname{Mov}}(X)\cap\Eff(X)$ under the action of the group $\operatorname{PsAut}(X)$ of birational automorphisms which are isomorphisms in codimension $1$. In its most general version formulated in \cite{Tot08}, the conjecture has the following form.

\begin{CC}
	Let $(X,\Delta)$ be a projective $\Q$-factorial klt pair such that $K_X+\Delta\equiv 0$.
	\begin{enumerate}[\normalfont (a)]
		\item There exists a rational polyhedral fundamental domain for the action of $\Aut(X)$ on $\Nef^e(X)$.
		\item There exists a rational polyhedral fundamental domain for the action of $\operatorname{PsAut}(X)$ on $\operatorname{Mov}^e(X)$.
	\end{enumerate}
\end{CC}

If $\Delta=0$ and $X$ is smooth, this reduces to a more classical version due to Morrison and Kawamata \cite{Mor93,Kaw97}. A detailed account on the history of the conjecture and its importance in difference contexts can be found in \cite{LOP18}. The two versions (a) and (b) of the Cone Conjecture are equivalent on surfaces.

This conjecture is known in full generality only on surfaces by \cite{Tot10}, and there are several partial results in higher dimensions. It is also explained in \cite{Tot08} why the conjecture, as stated, does not hold if the singularities of the pair $(X,\Delta)$ are worse than klt, or if the divisor $\Delta$ is not effective.

\medskip

\noindent{\sc Main results}

\medskip

A couple $(X,B+M)$ is a \emph{generalised pair} if $(X,B)$ is a usual projective pair and $M$ is a nef $\R$-divisor on $X$. We say that it is \emph{Calabi--Yau} if, additionally, $X$ is $\Q$-factorial and $K_X+B+M\equiv 0$, and it is klt if the pair $(X,B)$ is klt. This includes all smooth varieties with nef anticanonical class: indeed, if we set $M\coloneqq{-}K_X$, then $(X,M)$ is a klt Calabi--Yau generalised pair. Here comes the point: the generalised pair $(X,M)$ is klt, even though $X$ itself does not have necessarily a structure of a klt pair in the usual sense.

As mentioned above, the Cone Conjecture does not hold for the class of klt Calabi--Yau generalised pairs. However, for a very general member of that class, we prove that it \textbf{does hold} when we replace the automorphism group $\Aut(X)$ by the group of Cremona isometries $\Cr(X)$ as in Definition \ref{dfn:cremona}.

The main result of this paper and of \cite{LSX26} is the following.

\begin{thmA}\label{thm:main0}
	Let $(X,B+M)$ be a very general klt Calabi--Yau generalised pair of dimension $2$. Then there exists a rational polyhedral fundamental domain for the action of the group of Cremona isometries $\Cr(X)$ on $\Nef^e(X)$.
\end{thmA}

In Theorems \ref{thm:main1} and \ref{thm:main2} below we spell out explicitly what \emph{very general} means in the formulation of Theorem \ref{thm:main0}. Theorem \ref{thm:main0} should be understood as a shorthand for the union of Theorems \ref{thm:main1} and \ref{thm:main2}.

We start with Theorem \ref{thm:main1}, which is the main technical result of the current paper. Note that \emph{Sakai surfaces}, which are obtained by blowing up $\mathbb{P}^2$ at nine points and which appear in many of our results, are defined in Definition \ref{dfn:goodHalphen}. The term \emph{generic Sakai surfaces} is defined in Definition \ref{dfn:generic_Sakai}: by Proposition \ref{prop:noNodalOnGen1}, these surfaces are very general Sakai surfaces in the sense of moduli theory.

\begin{thmA}\label{thm:main1}
	Let $(X,B+M)$ be a klt Calabi--Yau generalised pair of dimension $2$. Then $\kappa(X,{-}K_X)\geq0$, and the following statements hold.
	\begin{enumerate}[\normalfont (a)]
		\item If $\kappa(X,{-}K_X)>0$ or if $X$ is not rational, then there exists a rational polyhedral fundamental domain for the action of the group $\Aut(X)$ on $\Nef^e(X)$.
		\item Assume that $\kappa(X,{-}K_X)=0$ and $X$ is rational. Let $D$ be the unique effective $\Q$-divisor such that $D\sim_\Q{-}K_X$. Then one of the following holds.
		\begin{enumerate}
			\item[$\mathrm{(b)}_1$] There exists a rational polyhedral fundamental domain for the action of the group $\Aut(X)$ on $\Nef^e(X)$.
			\item[$\mathrm{(b)}_2$] There exists a birational map $X\dashrightarrow X_0$ to a Sakai surface $X_0$.
		\end{enumerate}
		\item In the setup of $\mathrm{(b)}_2$, assume additionally that $X_0$ is a generic Sakai surface. Let $D_0$ be the unique element of the linear system $|{-}K_{X_0}|$. If there exists a rational polyhedral fundamental domain for the action of the group $\Cr(X_0,D_0)$ on $\Nef^e(X_0)$, then there exists a rational polyhedral fundamental domain for the action of the group $\Cr(X,D)$ on $\Nef^e(X)$.\footnote{The group $\Cr(X,D)$ is defined in Definition \ref{dfn:cremona}.}
	\end{enumerate}
\end{thmA}

Thus, in order to prove Theorem \ref{thm:main0} by using Theorem \ref{thm:main1}, it remains to resolve Theorem \ref{thm:main1}(c). This is where Theorem \ref{thm:main2} comes into the picture: it is one of the main results of the sequel to the present paper \cite{LSX26}.

\begin{thmA}\label{thm:main2}
	Let $X$ be a generic Sakai surface and let $D$ be the unique divisor in $|{-}K_X|$. Then there exists a rational polyhedral fundamental domain for the action of the group $\Cr(X,D)$ on $\Nef^e(X)$.
\end{thmA}

Finally, contrary to what one might expect after seeing Theorem \ref{thm:main2}, for 
\textbf{every} non-generic Sakai surface, no version of the Cone Conjecture with
respect to a geometric meaningful group holds. This is the content of the following Theorem \ref{thm:main3}, which will be proved in \cite{LSX26a}.

\begin{thmA}\label{thm:main3}
	Let $X$ be a non-generic Sakai surface. Let $G\subseteq\GL\big(N^1(X)_\R\big)$ be any group which preserves the intersection form on $N^1(X)_\R$ and the cone $\Eff(X)$. Then there does not exist a rational polyhedral fundamental domain for the action of $G$ on $\NefE(X)$.
\end{thmA}

\medskip

\noindent{\sc Generalised Cone Conjecture}

\medskip

In light of this paper and of its sequels \cite{LSX26,LSX26a}, 
we formulate the following conjecture.

\begin{GenCC}
	Let $(X,B+M)$ be a very general klt Calabi--Yau generalised pair. Then there exist subgroups $G_{\mathrm{nef}}$ and $G_{\mathrm{mov}}$ of $\GL\big(N^1(X)_\R\big)$ of geometric origin such that the following statements hold.
	\begin{enumerate}[\normalfont (a)]
		\item There exists a rational polyhedral fundamental domain for the action of $G_{\mathrm{nef}}$ on $\Nef^e(X)$.
		\item There exists a rational polyhedral fundamental domain for the action of $G_{\mathrm{mov}}$ on $\operatorname{Mov}^e(X)$.
	\end{enumerate}
\end{GenCC}

The minimal requirements for the term \emph{group of geometric origin} is that it preserves the intersection pairing and the respective cone, as in Proposition \ref{prop:groups}. Our Theorem \ref{thm:main0} proves the Generalised Cone Conjecture in dimension two.

\medskip

\noindent{\sc Related results}

\medskip

There are certain situations where a statement similar to that of the Cone Conjecture is known to hold. 

First, we recall that the nef cone of a Fano manifold $X$ is rational polyhedral and is contained in the effective cone, hence by \cite[Corollary 6.6]{Laz13} there exists a rational polyhedral fundamental domain for the action of $\Aut(X)$ on $\Nef(X)$. In other words, the Cone Conjecture holds on Fano varieties.

Second, under the weaker assumption that on a manifold $X$ the divisor $-K_X$ is semiample, there exists a klt Calabi--Yau pair structure on $X$. Then the Cone Conjecture is verified in certain cases, see for instance \cite{PS12}.

Third, the paper \cite{SX23} proves a version of the movable Cone Conjecture on the blowup $X$ of $\mathbb{P}^3$ at eight very general points: there exist a subgroup $\mathcal W\subseteq\GL\big(N^1(X)_\R\big)$ of geometric origin and a rational polyhedral fundamental domain for the action of $\mathcal W$ on $\operatorname{Mov}^e(X)$. Here, ${-}K_X$ is a nef divisor, but $X$ does not admit the structure of a klt Calabi--Yau pair, and the movable Cone Conjecture in the classical sense does not hold by \cite{Gac25}.

Finally, the study of \emph{Looijenga pairs}, which was initiated in \cite{Loo81} and fully developed in \cite{GHK15,Fri13,Fri15}, was a further source of inspiration for our work. A Looijenga pair $(X,D)$ consists of a smooth rational surface $X$ and a divisor $D\in|{-}K_X|$, where $D$ is either a rational nodal curve or a cycle of smooth rational curves. In \cite{Li25}, the following version of the Cone Conjecture was verified: given a very general Looijenga pair $(X,D)$, there exists a rational polyhedral fundamental domain for the action of a certain monodromy group on the effective nef cone. For a very general Looijenga pair $(X,D)$, this group coincides with the group $\Cr(X,D)$. 

\medskip

\noindent{\sc Relation to the Generalised Nonvanishing Conjecture}

\medskip

Our Generalised Cone Conjecture gives a strong structural prediction: for varieties as in the conjecture, the cones $\Nef^e(X)$ and $\operatorname{Mov}^e(X)$ are rational polyhedral up to the action of natural groups acting on them. This, together with the results of \cite{Loo14}, implies that every nef line bundle on $X$ belongs to the effective cone. This prediction is known as the Generalised Nonvanishing Conjecture \cite{HanLiu20,LP20a}, and was shown in \cite{HanLiu20} on surfaces and almost completed on threefolds in \cite{LMPTX23}.

\medskip

\noindent{\sc Why are our results surprising?}

\medskip

We now explain why the results of this paper and of its sequel \cite{LSX26} have previously been unexpected.

Recall the following example from \cite{Tot10}. If $X$ is the blow up of $\mathbb{P}^2$ at nine very general points, then $M\coloneqq{-}K_X$ is nef but not semiample. The surface $X$ does not admit a structure of a klt Calabi--Yau pair, but $(X,M)$ is a klt Calabi--Yau generalised pair. As $X$ contains infinitely many $(-1)$-curves by \cite[Theorem 4a]{Nag61}, the cone $\overline{\Eff}(X)$ is not rational polyhedral. By the duality given by the intersection product, the cone $\Nef(X)$ is also not rational polyhedral, hence neither is $\Nef^e(X)$. However, the group $\Aut(X)$ is trivial by \cite[Proposition 8]{Giz80}. Thus, the classical Cone Conjecture fails as badly as possible in this example, but \cite{DU24} showed that it nevertheless holds after replacing the group $\Aut(X)$ by the group of Cremona isometries $\Cr(X)$. This is a special case of our main result.

As mentioned above, similar behaviour was demonstrated in \cite{SX23} on the blowup of $\mathbb{P}^3$ at eight very general points.

In Theorem \ref{thm:main0} it is surprising that there exists \emph{any other group}, apart from the automorphism group, whose action on the nef cone has a rational polyhedral fundamental domain, not only in the class of surfaces with nef anticanonical class, but even in a more general context.

After Theorem \ref{thm:main0} is established, it is then quite surprising that its statement does not hold for \emph{all} klt Calabi--Yau generalised pairs of dimension $2$: it is very much not obvious why special $({-}2)$-curves would cause arguments to break down. The results \cite[Theorem 5.1]{GHK15}, \cite[Remark 9.2]{Fri15} and \cite[Theorem 1.1(2)]{Li25} hint at the existence of an appropriate group also in the case of non-generic Sakai surfaces. Therefore, finding such a counterexample in Theorem \ref{thm:main3} was very much unexpected for us.

\medskip

\noindent{\sc What is in this paper}

\medskip

We explain briefly the content of the paper.

This paper contains the proof of Theorem \ref{thm:main1}.
We also prove lifting and descent results, Theorems \ref{thm:liftingoneblowup} and \ref{thm:descent}, which are interesting in their own right and which form a big part of the proof of Theorem \ref{thm:main1}. In this paper we treat Theorem \ref{thm:main2} as a black box and we prove Theorems \ref{thm:main0} and \ref{thm:main1} assuming Theorem \ref{thm:main2}.

Along the way we prove several results which are highly relevant in their own right. In Theorem \ref{thm:nefandeffectivecones} we prove an essential result on the relation between the cones $\Nef(X)$ and $\Nef^e(X)$ on a klt Calabi--Yau generalised pair of dimension two. Corollary \ref{cor:nonrational} gives a concrete characterisation of smooth Calabi--Yau generalised pairs of dimension two which are not rational.

\medskip

\noindent{\sc What is in \cite{LSX26} and \cite{LSX26a}}

\medskip

The sequels to this paper proves Theorems \ref{thm:main2} and \ref{thm:main3}. The methods of proofs of those two theorems are completely different from those of the present paper. Whereas in this paper we mainly use techniques from birational geometry, \cite{LSX26} and \cite{LSX26a} rely on the machinery of root systems and Weyl groups.

\subsection*{Acknowledgements} 

We would like to thank C\'ecile Gachet and Wendelin Lutz for many useful discussions and comments over several years on the content of the present paper. We are grateful to Pierre-Emmanuel Chaput and Julia Schneider for many valuable discussions. This project started in 2023 when all three authors were based at the Universit\"at des Saarlandes.

Lazi\'c gratefully acknowledges support by the Deutsche Forschungsgemeinschaft (DFG, German Research Foundation) – Project-ID 286237555 – TRR 195 and Project-ID 530132094. The research of Stenger was partly conducted in the framework of the DFG-funded research training group RTG 2965: From Geometry to Numbers, Project number 512730679.
Xie was partly supported by PEPS JCJC de l'Insmi 2025.

\section{Preliminaries}\label{sec:prelim}

Throughout the paper, unless otherwise stated, we work with varieties that are normal, projective and defined over $\C$.

We frequently use the Negativity lemma \cite[Lemma 3.39]{KM98}.

	The \emph{Iitaka dimension} of a $\Q$-Cartier divisor $D$ on a normal variety $X$ is denoted by $\kappa(X,D)$. If $D$ is also nef, the \emph{numerical dimension} of $D$ is 
	$$\nu(X,D)\coloneqq\sup\{k\in\N\mid D^k\not\equiv0\}\in\{0,\dots,\dim X\} . $$
	It is well known that $\kappa(X,D)\leq\nu(X,D)$, and that $D$ is big if and only if $\nu(X,D)=\dim X$. We will not need other properties of the numerical dimension in this paper.
		
If $X$ is a normal projective variety, we denote by $N^1(X)_\R$, $\Nef(X)$, $\Amp(X)$, $\Eff(X)$ and $\overline{\Eff}(X)$ the N\'eron-Severi space and the nef, ample, effective and pseudoeffective cone of $X$, respectively. If $D$ is an $\R$-Cartier $\R$-divisor on $X$, then we denote its numerical class in $N^1(X)_\R$ by $[D]$. We call the cone
$$\Nef^e(X)\coloneqq\Nef(X)\cap\Eff(X)$$
the \emph{effective nef cone}. The \emph{rational nef cone} $\Nef^+(X)$ is the cone in $N^1(X)_\R$ spanned by the classes of nef line bundles. The cones $\Nef^e(X)$ and $\Nef^+(X)$ are the most important cones related to the (Generalised) Cone Conjecture. The relationships between them are explored in Section \ref{sec:CY2}.

	Let $D$ be an $ \R $-Cartier $\R$-divisor on a normal projective variety $X$. Following \cite{LP20a,LP20b}, we say that $D$ is \emph{num-effective} if $[D]\in\Eff(X)$, and that $D$ is \emph{num-semiample} if $D$ is numerically equivalent to a semiample $\R$-divisor on $X$.

\subsection{Pullbacks and pushforwards}

If $f\colon X\to X'$ is a birational morphism between normal projective varieties, then it is clear that $f^*\Eff(X')\subseteq\Eff(X)$ and $f^*\Nef(X')\subseteq\Nef(X)$. Considering pushforwards is much more subtle. What is not clear is that one may define the pushforward map compatible with the numerical equivalence. However, this holds if $X'$ is $\Q$-factorial. We will use the following easy lemma often in the paper.

\begin{lem}\label{lem:pushforward}
Let $f\colon X\to X'$ be a birational morphism between $\Q$-factorial projective varieties. Let $E_1,\dots,E_k$ be all the $f$-exceptional prime divisors on $X$, and set $V\coloneqq\sum_{i=1}^k \R [E_i]\subseteq N^1(X)_\R$. Then the following statements hold.
\begin{enumerate}[\normalfont (a)]
\item Let $D$ be a $\Q$-divisor on $X$. Then $\kappa(X,D)\leq\kappa(X',f_*D)$.
\item There exists a well-defined map $f_*\colon N^1(X)_\R\to N^1(X')_\R$, and we have $f_*\Eff(X)\subseteq\Eff(X')$.
\item Then $ N^1(X)_\R = f^* N^1(X')_\R \oplus V $ and $V=\bigoplus_{i=1}^k \R [E_i]$. If $x\in N^1(X)_\R$, then there exists $e\in V$ such that $x=f^*f_*x+e$.
\end{enumerate}
Assume, furthermore, that $X$ and $X'$ are surfaces.
\begin{enumerate}[\normalfont (d)]
\item If $N$ is a nef $\R$-divisor on $X$, then $f_*N$ is a nef $\R$-divisor on $X'$, and consequently $f_*\Nef(X)\subseteq\Nef(X')$.
\item[\normalfont{(e)}] For any $e\in V$ we have $e^2\leq0$, with equality if and only if $e=0$.
\item[\normalfont{(f)}] If $N$ is a nef $\R$-divisor on $X$, then there exists an $f$-exceptional $\R$-divisor $E\geq0$ on $X$ such that $N+E=f^*f_*N$.
\item[\normalfont{(g)}] If $D$ is an $\R$-divisor on $X$ such that $D^2=(f_*D)^2$, then $D=f^*f_*D$.
\end{enumerate}
\end{lem}

\begin{proof}
Most of these statements are well known, and we recall the proofs for the benefit of the reader.

Part (a) follows from \cite[Lemma 2.8]{LP18a}. Part (d) follows from Kleiman's criterion for nefness and the projection formula. Part (f) is an immediate consequence of the Negativity lemma.

For (b), let $D_1$ and $D_2$ be two $\R$-Cartier $\R$-divisors on $X$ such that $D_1\equiv D_2$. Then \cite[Lemma 3.4]{LMPTX23}, applied to the divisor $D_1-D_2$, shows that $f_*D_1\equiv f_*D_2$. This shows that we may define the map $f_*\colon N^1(X)_\R\to N^1(X')_\R$. The inclusion $f_*\Eff(X)\subseteq\Eff(X')$ is then obvious.

Now we prove (c). If $D$ be an $\R$-divisor on $X$, then clearly there exists an exceptional $\R$-divisor $F$ such that $D=f^*f_*D+F$. This shows the last statement in (c), and that $N^1(X)_\R=f^* N^1(X')_\R+V$. Now, if for some $x'\in N^1(X')_\R$ and some real numbers $\lambda_i$ we have $f^*x'+\sum_{i=1}^k\lambda_i[E_i]=0$, then by pushing forward by $f_*$ we immediately get that $x'=0$. But then $\lambda_i=0$ for all $i$ by the Negativity lemma. This completes the proof of (c).

For (e), by passing to a resolution we may assume that $X$ is smooth. If $h$ is an ample class on $X'$, then $f^*(h)\cdot e = 0$ and $(f^*h)^2>0$, hence the claim follows from the Hodge index theorem.

Finally, we show (g). There exists an $\R$-divisor $E$ on $X$ such that $D=f^*f_*D+E$. Then
$$D^2=(f^*f_*D+E)^2=(f^*f_*D)^2+2f^*f_*D\cdot E+E^2=D^2+E^2,$$
hence $[E]=0$ by (e). But then $E=0$ by (c).
\end{proof}

\subsection{Zariski decomposition}\label{subsec:Zariski}

Let $X$ be a smooth projective surface. Then for every pseudoeffective $\R$-divisor $D$ on $X$ there exists a unique decomposition $D=N+P$ into a \emph{negative} part $N$ and \emph{positive} part $P$ such that $P$ and $N$ are $\R$-divisors, $P$ is nef, the intersection of $P$ with any component of $N$ is zero, and the intersection matrix of the components of $N$ is negative definite. Moreover, if $D\geq0$, then also $P\geq0$.

\begin{rem}\label{rem:zariski}
When $D$ above is a $\Q$-divisor, then the existence of the Zariski decomposition is classical \cite{Zar62,Fuj79}. When $D$ is an $\R$-divisor, then the situation is more complicated. There are two approaches: an algebraic approach by Nakayama in \cite{Nak04}, and an analytic approach by Boucksom in \cite{Bou04}, which coincide on projective manifolds. Nakayama shows in \cite[remarks on p.~96]{Nak04} that his approach generalises Zariski's construction for $\Q$-divisors on smooth projective surfaces. Boucksom's shows in \cite[Proposition 3.6 and Theorems 4.5 and 4.8]{Bou04} that his approach generalises Zariski's, and satisfies the properties stated before this remark.
\end{rem}

When $X$ is only $\Q$-factorial and $D$ is a pseudoeffective $\R$-divisor on $X$, let $\pi\colon X'\to X$ be a resolution of $X$, and let $\pi^*D=N'+P'$ be the Zariski decomposition of $\pi^*D$. We define the \emph{Zariski decomposition $D=N+P$ of $D$} by setting $N\coloneqq\pi_*N'$ and $P\coloneqq\pi_*P'$. Then we have the following lemma.

\begin{lem}\label{lem:Zariski}
Let $X$ be a $\Q$-factorial projective surface, let $D$ be a pseudoeffective $\R$-divisor on $X$, and let $D=N+P$ be the Zariski decomposition of $D$. Then:
\begin{enumerate}[\normalfont (a)]
\item $N$ is an effective $\R$-divisor and $P$ is a nef $\R$-divisor,
\item $N$ and $P$ do not depend on the choice of a resolution of $X$,
\item if $D=B+M$, where $B$ is an effective $\R$-divisor and and $M$ is a nef $\R$-divisor, then $B\geq N$,
\item if $D\geq0$, then also $P\geq0$,
\item if, moreover, $D$ is a $\Q$-divisor, then $N$ and $P$ are also $\Q$-divisors, and we have $\kappa(X,D)=\kappa(X,P)$.
\end{enumerate}
\end{lem}

\begin{proof}
Parts (a) and (d) are clear by the discussion before the lemma, and by Lemma \ref{lem:pushforward}(d). Part (b) follows since Nakayama's decomposition mentioned in Remark \ref{rem:zariski} is compatible with pushforwards under birational morphisms by \cite[Lemma 2.13]{LX25}.

For (c), let $\pi\colon X'\to X$ be a resolution of $X$ and let $\pi^*D=N'+P'$ be the Zariski decomposition of $\pi^*D$. Since $\pi^*D=\pi^*B+\pi^*M$, we have $\pi^*B\geq N'$ by \cite[Proposition III.1.14(b)]{Nak04}. Then we obtain (c) by applying $\pi_*$ to this last inequality.

Finally, we show (e). The rationality of $N$ and $P$ follows from the constructions in \cite{Zar62,Fuj79}. Now, since $D=N+P$ and $N\geq0$, we have $\kappa(X,D)\geq\kappa(X,P)$. For the reverse inequality,  let $\pi\colon X'\to X$ be a resolution of $X$. Then we observe that by Lemma \ref{lem:pushforward}(a) and \cite[Proposition 2.3.21]{Laz04} we have $\kappa(X,P)\geq\kappa(X',P')=\kappa(X,\pi^*D)=\kappa(X,D)$. This completes the proof.
\end{proof}

\subsection{Curves with negative self-intersection}

In this paper, we need a precise description of curves with negative self-intersection on a projective surface $X$. We start with the following lemma.

\begin{lem}\label{lem:negativeexceptional}
Let $X$ be a smooth projective surface and let $C$ be an irreducible curve on $X$ such that $C^2 <0$. If $C'$ is an effective $\R$-divisor on $X$ such that $C\equiv C'$, then $C=C'$.
\end{lem}

\begin{proof}
The curve $C$ is exceptional in the sense \cite[Definition 3.10]{Bou04} by \cite[Theorem 4.5]{Bou04}. Then we conclude by \cite[Proposition 3.13]{Bou04}.
\end{proof}

The proof of the next result is similar to that of \cite[Lemma 1.2]{AM04}.

\begin{lem}\label{lem:negativeselfintersection}
    Let $X$ be a smooth projective surface, let $B\geq0$ be an $\R$-divisor on $X$ such that ${-}(K_X+B)$ is nef, and let $C$ be an irreducible curve on $X$ which is not a component of $B$ and such that $C^2<0$. Then the following statements hold.
	\begin{enumerate}[\normalfont (a)]
	\item We have $C^2\geq{-}2$ and $C\simeq\mathbb{P}^1$.
	\item If $C^2={-}2$, then $C$ is disjoint from $\Supp B$ and $(K_X+B)\cdot C=0$.
	\item If $C^2={-}1$, then $C$ is a $(-1)$-curve and $B\cdot C\leq1$. If $K_X+B\equiv0$, then $B\cdot C=1$.
	\end{enumerate}
\end{lem}

\begin{proof}
By the adjunction formula we have
\begin{align}
{-}2&\leq 2p_a(C)-2=(K_X+C)\cdot C\label{eq:p_a}\\
&\leq (K_X+C)\cdot C+B\cdot C=(K_X+B)\cdot C+C^2\leq C^2,\notag
\end{align}
which shows the first statement in (a). Since $C^2<0$ and since $2p_a(C)-2$ is even, we must have $p_a(C)=0$, hence $C\simeq\mathbb{P}^1$. This shows (a).

If $C^2={-}2$, then from the second inequality in \eqref{eq:p_a} we obtain that $B\cdot C=0$, and from the third inequality in \eqref{eq:p_a} we obtain that $(K_X+B)\cdot C=0$, which gives (b).

If $C^2={-}1$, then from \eqref{eq:p_a} we obtain that $K_X\cdot C={-}1$ and $B\cdot C\leq1$. If $K_X+B\equiv0$, then \eqref{eq:p_a} gives $B\cdot C=1$, which yields (c).
\end{proof}

Thus, if $X$ is a smooth projective surface and if $B\geq0$ is an $\R$-divisor on $X$ such that ${-}(K_X+B)$ is nef, then Lemma \ref{lem:negativeselfintersection} says that a curve $C$ on $X$ with $C^2={-}2$ is either a component of $B$ or it is a smooth rational curve disjoint from $\Supp B$. This motivates the following definition.

\begin{dfn}\label{dfn:nodal_set}
Let $X$ be a smooth projective surface and let $B\geq0$ be an $\R$-divisor on $X$. We set
$$\Delta_{X,B}^{\nod}\coloneqq\big\{[C]\mid C\text{ is a curve on $X$ with $C^2={-}2$},\ C\cap\Supp B=\emptyset\big\}.$$
This is a well-defined set by Lemma \ref{lem:negativeexceptional}. A curve whose class is in $\Delta_{X,B}^{\nod}$ is called an \emph{internal $(-2)$-curve}.
\end{dfn}

\subsection{Sakai surfaces}

Let $D = \sum n_iD_i$ be an effective integral divisor on a smooth projective surface $X$, where $D_i$ are prime components. Then $D$ is \emph{of canonical type} if $K_X\cdot D_i = D\cdot D_i = 0$ for all $i$.

Recall that a \emph{Halphen surface} is a smooth rational surface if there exists a positive integer $m$ such that the linear system $|{-}mK_X|$ is basepoint free and of dimension $1$. The smallest such positive integer $m$ is called the \emph{index of $X$}. One can easily show that $K_X^2=0$ and that there exists an elliptic fibration $X\to\mathbb P^1$.

The following definition, introduced in \cite{Sak01}, relaxes the notion of Halphen surfaces.

\begin{dfn}\label{def:gen_halphen_surface}
	A smooth rational surface $X$ is a \emph{generalised Halphen surface} if the linear system $|{-}K_X|$ contains a divisor of canonical type. If, additionally, we have $h^0(X,-K_X)=1$, then $X$ is called an \emph{$R$-surface}, where $R$ is the type of the unique divisor $D \in |{-}K_X|$.
\end{dfn}

If $X$ is an $R$-surface, then the intersection matrix of $D \in |{-}K_X|$ determines a root system $R$ of affine type: such surfaces are then classified according to the \emph{type $R$} of $D$, and we also say that $X$ is \emph{of type $R$}.
The following table is \cite[Table 2, p.\ 181]{Sak01}, listing the possible types of $R$-surfaces.

\begin{table}[h!]
\begin{center}
\caption{Classification of the types of $R$-surfaces}
\begin{tabular}{|p{0.26\linewidth}|p{0.65\linewidth}|}
  \hline
  \quad & Type $R$ \\ 
  \hline
  Elliptic type & $A_0^\smone$ \\ 
  \hline
  Multiplicative type & $A_0^\smonestar, A_1^\smone,\ldots, A_7^\smone, {A_7^\smone}', A_8^\smone$ \\ 
  \hline
  Additive type & $A_0^\smonestarstar, A_1^\smonestar, A_2^\smonestar, D_4^\smone, \ldots, D_8^\smone, E_6^\smone, E_7^\smone, E_8^\smone$ \\
  \hline
\end{tabular}
\label{table:typeR}
\end{center}
\end{table}

In this paper and in its sequel \cite{LSX26}, the following refinement of generalised Halphen surfaces plays a central role.

\begin{dfn}\label{dfn:goodHalphen}
	A smooth projective surface $X$ is a \emph{Sakai surface} if it is an $R$-surface with $\kappa(X,-K_X) = 0$.
\end{dfn}

The following result shows that if $X$ is a surface of type $R$, then $X$ is either a Sakai surface or a Halphen surface of index $m>1$. Moreover, if $R$ is of additive type, then it is necessarily a Sakai surface.

\begin{lem}\label{lem:char_of_indecomp}
	Let $X$ be a generalised Halphen surface. Then:
	\begin{enumerate}[\normalfont (a)]
		\item $-K_X$ is nef and $K_X^2 = 0$,
        \item $X$ is the blowup of $\pr^2$ at nine (possibly infinitely near) points,
		\item $X$ is either a Sakai surface or a Halphen surface,
		\item if $X$ is an $R$-surface, then either $X$ is a Sakai surface, or $X$ is a Halphen surface of index $m>1$ and $X$ is of elliptic type or of multiplicative type.
	\end{enumerate}
\end{lem}

\begin{proof}
	Let $D\in|{-}K_X|$ be a divisor of canonical type, and let $C$ be an irreducible curve on $X$. If $C$ is not a component of $D$, then ${-}K_X\cdot C=D\cdot C\geq0$. If $C$ is a component of $D$, we have $D\cdot C=0$ since $D$ is of canonical type. This shows that ${-}K_X$ is nef, and $K_X^2 = D^2 = 0$, which proves (a).

    Part (b) follows immediately from \cite[Proposition 2(a)(d)]{Sak01}.
	
	Now we show (c). Note first that $-K_X$ not big by (a), and $|{-}K_X|\neq\emptyset$ by assumption. Thus, $\kappa(X,-K_X)\in\{0,1\}$. If $\kappa(X,-K_X) =0 $, then $X$ is a Sakai surface by definition. If $\kappa(X,-K_X) =1 $, then $-K_X$ is semiample, see for instance \cite[Lemma 3.3]{AL11}. Hence, $X$ is a Halphen surface.

	Finally, we show (d). As in the proof of (c), if $\kappa(X,-K_X) = 0$, then $X$ is a Sakai surface, hence we may assume that $\kappa(X,-K_X) = 1$.  Then by (c), the surface $X$ is a Halphen surface of index $m$. Since $|{-}K_X|$ is of dimension $1$, we conclude that $m> 1$. Let $f\colon X\to\mathbb P^1$ be the induced elliptic fibration. Since $f$ is given by the linear system $|{-}mK_X|$, the divisor $mD$ is a multiple fibre of $f$, and it is the unique multiple fibre of $f$ by \cite[Proposition 2.2]{CD12}. Then \cite[Proposition 5.1.8 on p.\ 291]{CD89} implies that the type of $D$ is either elliptic or multiplicative.
\end{proof}

We have the following important result which will be used in the proof of Theorem \ref{thm:main1}.

\begin{thm}\label{thm:sak84}
Let $X$ be a smooth rational surface with $\kappa(X,-K_X) = 0$ and let $D$ be the unique effective $\Q$-divisor such that $D\sim_\Q{-}K_X$. Let $P$ be the positive part in the Zariski decomposition of $D$, and assume that $P\neq0$. Then there exists a birational morphism $\pi\colon X \rightarrow X_0$ to a smooth surface $X_0$, such that $\pi$ is the composition of contractions of $(-1)$-curves and:
\begin{enumerate}[\normalfont (a)]
\item $X_0$ is a Sakai surface,
\item there exists $\alpha\in\Q_{>0}$ such that $ P =\alpha\pi^*D_0 $, where $D_0\in|{-}K_{X_0}|$,
\item $\Exc(\pi)\subseteq\Supp P\subseteq\Supp D$,
\item $D\in|{-}K_X|$,
\item $D$ is connected,
\item $\Delta_{X,D}^{\nod}=\emptyset$ if and only if $\Delta_{X_0,D_0}^{\nod}=\emptyset$.
\end{enumerate}
\end{thm}

\begin{proof}
The first statement of the theorem, as well as (a), is contained in \cite[Theorem 3.4]{Sak84}. Also by op.\ cit.\ there exists a positive rational number $\alpha$ such that $ P \sim_\Q\alpha\pi^*D_0 $, and then (b) follows since $ \kappa(X,-K_X) = 0$ and since $P\geq0$ by Lemma \ref{lem:Zariski}(d).

Now, since $X_0$ is smooth, there exists an effective $\pi$-exceptional divisor $E$ containing all curves contracted by $\pi$, such that
\begin{equation}\label{eq:67aaa}
{-}K_X+E \sim \pi^*({-}K_{X_0}).
\end{equation}
Since $D\sim_\Q{-}K_X$ and $D_0\sim{-}K_{X_0}$, we conclude that $D+E \sim_\Q \pi^*D_0$. As $\kappa(X,\pi^*D_0)=\kappa(X_0,D_0)=0$ and $D+E \geq 0$, this yields
\begin{equation}\label{eq:67}
D+E = \pi^*D_0.
\end{equation}
Therefore, $\Supp E\subseteq \Supp \pi^*D_0=\Supp P $ by (b). As $0\leq P\leq D$ by Lemma \ref{lem:Zariski}(d), we deduce (c). Furthermore, since $E$ and $D_0$ are integral divisors, we conclude from \eqref{eq:67} that $D$ is integral as well. Moreover, \eqref{eq:67aaa} gives
$$-K_X\sim \pi^*D_0-E=D,$$
which proves (d). Part (e) is now proved in the same way as \cite[Proposition 6]{Sak01}.

Finally, by (c) and (e) we conclude that $\pi\big(\Exc(\pi)\big)\subseteq\Supp D_0$, as $D_0 = \pi_* D$. Therefore, the map $\pi$ induces an isomorphism between the sets $X\setminus\Supp D$ and $X_0\setminus\Supp D_0$. It is easy to see that this implies (f), and finishes the proof.
\end{proof}

\begin{rem}
The map $\pi$ from Theorem \ref{thm:sak84} is not an isomorphism in general. For instance, let $X_0$ be a Sakai surface such that $-K_{X_0}\sim D_1+D_2$, where $D_1$ and $D_2$ are smooth $({-}2)$-curves such that $D_1\cdot D_2=2$; such a surface exists by \cite[Appendix A]{Sak01}. Let $\pi\colon X\to X_0$ be the blowup of $X_0$ along the two points in $D_1\cap D_2$. Then $X$ satisfies the assumptions of Theorem \ref{thm:sak84}. Indeed, let $D_1'$ and $D_2'$ be the strict transforms of $D_1$ and $D_2$ on $X$, and set $B\coloneqq\frac12(D_1'+D_2'+\pi^*D_1+\pi^*D_2)$. Then ${-}K_X\sim B$, hence $\kappa(X,{-}K_X)\geq0$. Moreover, $\kappa(X,{-}K_X)\leq\kappa(X_0,{-}K_{X_0})=0$ by Lemma \ref{lem:pushforward}(a), thus $\kappa(X,{-}K_X)=0$. If $P$ is the positive part in the Zariski decomposition of $B$, then $\kappa(X,P)=0$ by Lemma \ref{lem:Zariski}(e). In particular, $P^2=0$ since $P$ is not big, and $P\neq0$ since $P\geq \frac12\pi^*(D_1+D_2)$ by Lemma \ref{lem:Zariski}(c).
\end{rem}

\subsection{The effective cone}

If $X$ is a smooth projective surface and if $H$ is an ample $\Q$-divisor on $X$, define
$$ \mathcal C_X^+ \coloneqq  \{ x\in N^1(X)_\R\mid x^2 >0  \text{ and } x\cdot H >0 \}.$$
This cone does not depend on $H$: indeed, fix $x\in N^1(X)_\R$ such that $x^2 >0$. Let $H_1$ and $H_2$ be ample $\Q$-divisors and assume that $x\cdot H_1>0$ and $x\cdot H_2\leq0$. Then there exists $t\geq0$ such that $x\cdot(tH_1+H_2)=0$. Since $tH_1+H_2$ is ample, we conclude by the Hodge index theorem that $x=0$, a contradiction.

If $X$ is a smooth projective surface, we further define
$$ \cal{M}_X \coloneqq  \big\{ [E]\in N^1(X) \mid E^2 = K_X\cdot E = -1 \text{ and }h^0(X,E)>0 \big\}. $$
We have the following very useful lemma.

\begin{lem}\label{lem:rational_Mori_cone}
Let $X$ be a smooth rational surface and let $B\geq0$ be an $\R$-divisor such that ${-}(K_X+B)$ is nef. Let $\Comp(B)$ be the set of numerical classes of all the components of $B$. Then $\overline{\Eff}(X)$ is the convex hull of
    $$\overline{\mathcal C_X^+}\cup\mathcal{M}_X\cup\Comp(B)\cup\Delta_{X,B}^{\nod}.$$
\end{lem}

\begin{proof}
	We have $N^1(X)\cap\mathcal C_X^+ \subseteq \Eff(X)$ by \cite[Corollary V.1.8]{Har77}. Consequently, $N^1(X)_\Q\cap\mathcal C_X^+ \subseteq \Eff(X)$, hence $\overline{\mathcal C_X^+} \subseteq \overline{\Eff}(X)$. Since clearly $\mathcal M_X\cup\Comp(B)\cup\Delta_{X,B}^{\nod}\subseteq\overline{\Eff}(X)$, this shows one inclusion.
	
	For the opposite inclusion, let $D$ be a pseudoeffective $\R$-divisor on $X$, and let $D=N+P$ be the Zariski decomposition of $D$. Then $[P]\in\overline{\mathcal C^+_X}$ since $P$ is nef. Furthermore, the intersection matrix of the components of $N$ is negative definite. If $C$ is a component of $N$, this implies in particular that $C^2<0$. Then by Lemma \ref{lem:negativeselfintersection} we have $[C]\in \mathcal{M}_X\cup\Comp(B)\cup\Delta_{X,B}^{\nod}$. Now the lemma follows immediately.
    \end{proof}

We will also need a more precise characterisation of the set $\mathcal M_X$.

\begin{lem}\label{lem:rationalsurface-1curve}
    Let $X$ be a smooth projective rational\footnote{In the proof of Lemma \ref{lem:rationalsurface-1curve} we do not use rationality explicitly. We use the following two facts: that ${-}K_X$ is pseudoeffective and that $\chi(X,\OO_X)\geq1$. However, any smooth projective surface satisfying these two conditions is necessarily rational. Indeed, otherwise its relatively minimal model is a ruled surface $X'$ over a curve of genus $g\geq1$. Since both $X$ and $X'$ have rational singularities, we have $\chi(X,\OO_X)=\chi(X',\OO_{X'})$, and $\chi(X',\OO_{X'})=1-g\leq0$ by \cite[Corollary V.2.5]{Har77}, a contradiction.} surface such that ${-}K_X$ is pseudoeffective.
    \begin{enumerate}[\normalfont (a)]
    \item Let $E$ be an integral divisor on $X$ such that $ E^2 = K_X\cdot E = {-}1 $. If $h^0(X,E)>0$, then $E\cdot H > 0$ for any ample divisor $H$ on $X$. If $E\cdot H > 0$ for some ample divisor $H$ on $X$, then $h^0(X,E)>0$.
     \item For any ample divisor $H$ on $X$ we have
    $$ \cal{M}_X = \big\{ [E]\in N^1(X) \mid E^2 = K_X\cdot E = -1 \text{ and }E\cdot H >0 \big\}. $$
    \end{enumerate}    
\end{lem}

\begin{proof}
    Part (b) follows immediately from (a). The first statement in (a) is clear. For the second statement in (a), let $H$ be an ample divisor on $X$ such that $E\cdot H>0$. Since ${-}K_X$ is pseudoeffective, we have ${-}K_X\cdot H\geq0$, hence $(K_X-E)\cdot H<0 $. This implies that $h^0(X,K_X - E)=0$, hence $h^2(X,E)=0$ by Serre duality. Therefore, Riemann--Roch yields
    $$ \textstyle h^0(X,E) \geq \chi\big(X,\OO_X(E)\big)=1 + \frac{1}{2} E\cdot (E-K_X) = 1. $$
    This concludes the proof.
\end{proof}
	
\subsection{Pairs and generalised pairs}\label{subsec:pairs}

	A \emph{pair} $(X,B)$ consists of a normal variety $X$ and an effective Weil $\R$-divisor $B$ on $ X $ such that the divisor $K_X+B$ is $\R$-Cartier. We use freely notions of pairs and singularities of pairs as in \cite{KM98}. However, we only use the MMP occasionally, and even then only in the classical sense; a clear exposition of this theory is in \cite{Mat02}.
	
		If $X$ is a $\Q$-factorial klt variety, then we denote by $N_W^1(X)$ the group consisting of numerical classes of Weil divisors on $X$. Then by \cite[Theorem 1.10]{GKP16a}, there exists a positive integer $m$ and a well-defined injective map $N_W^1(X)\to N^1(X),\ [D]\to[mD]$. Since $N^1(X)$ is a lattice in $N^1(X)_\R$, we conclude that $N_W^1(X)$ is a lattice in $N^1(X)_\R$, see also \cite[Remark after Lemma 11]{Xu24}.
	
	We will need the following result crucially several times in the paper.

\begin{thm}\label{thm:terminalisation}
	Let $(X,B)$ be a klt pair. Then there exists a \emph{$\Q$-factorial terminalisation of $(X,B)$}, i.e.\ a $\Q$-factorial terminal pair $(X',B')$ together with a proper birational morphism $f\colon X' \to X$ such that
	$$ K_{X'}+B'\sim_\R f^*(K_X+B). $$
	If $X$ is a surface, such a pair $(X',B')$ is unique up to isomorphism, and $X'$ is smooth. 
\end{thm}

\begin{proof}
The existence follows from \cite[Corollary 1.4.3]{BCHM} and the paragraph after that result. If $X$ is a surface, then $X'$ is smooth by \cite[Theorem 4.5(1)]{KM98}.

Let $g\colon (X'',B'')\to(X,B)$ be another $\Q$-factorial terminalisation of the pair $(X,B)$, and denote $\varphi\coloneqq g^{-1}\circ f$. We claim that $\varphi$ does not contract a divisor. Assume, for contradiction, that $E$ a prime divisor on $X'$ which is contracted by $\varphi$. Then $a(E,X',B')=-\mult_E B'\leq 0$, whereas $a(E,X'',B'')>0$ since $(X'',B'')$ is terminal and $E$ is exceptional over $X''$. On the other hand, since $K_{X'}+B'\sim_{\R} f^*(K_X+B)$ and $K_{X''}+B''\sim_{\R} g^*(K_X+B)$, we conclude that $a(E,X',B')=a(E,X,B)=a(E,X'',B'')$. This is a contradiction which proves the claim. Analogously, the map $\varphi^{-1}$ does not contract a divisor. Therefore, $\varphi$ is an isomorphism in codimension one, thus an isomorphism since $X$ is a surface.
\end{proof}

In this paper we adopt the following definition of Calabi--Yau generalised pairs.

\begin{dfn}\label{dfn:g-pairs}
    Let $(X,B)$ be a projective pair and let $M$ be a nef $\R$-divisor on $X$. We call the couple $(X,B+M)$ a \emph{generalised pair}, and we say that it is klt, respectively log canonical, if $(X,B)$ is klt, respectively log canonical. If, moreover, $X$ is $\Q$-factorial and
    $$K_X+B+M\equiv 0,$$
    then we say that $(X,B+M)$ is a \emph{Calabi--Yau generalised pair}.
\end{dfn}

A few remarks are in order. The terminology comes from the theory of generalised pairs introduced in \cite{BH14,BZ16}. In that theory, the divisor $M$ does not have to be $\R$-Cartier, it is only the pushforward of a nef $\R$-divisor on a higher birational model, and the singularities are defined slightly differently. However, if $X$ is a surface, then the definition of generalised pairs from op.\ cit.\ coincides with Definition \ref{dfn:g-pairs} by \cite[Lemma 2.4]{HanLiu20}.

\subsection{Cremona isometries}

The following definition is similar to that in \cite[Section 4]{Dol83}, see also \cite{Loo81}.

\begin{dfn}\label{dfn:cremona}
    Let $X$ be a $\Q$-factorial klt projective surface. An automorphism $\sigma$ of the vector space $N^1(X)_\R$ is a \emph{Cremona isometry} of $X$ if it satisfies the following properties:
    \begin{enumerate}[\normalfont (i)]
    	\item $\sigma$ maps the lattice $N_W^1(X)$ onto itself,\footnote{Recall that $N_W^1(X)$ is a lattice by \S\ref{subsec:pairs}.}
        \item $\sigma$ preserves the intersection form on $N^1(X)_\R$,
        \item $\sigma$ fixes the class $[K_X]$,
        \item $\sigma$ preserves the effective cone $\Eff(X)$.
    \end{enumerate}
    We denote the group of Cremona isometries of $X$ by $\Cr(X)$. Additionally, if $D$ is an effective divisor on $X$ with irreducible components $D_1,\dots,D_m$, we define $\Cr(X,D)$ as
    \[\Cr(X,D)\coloneqq\big\{\sigma\in\Cr(X)\mid\sigma\big([D_i]\big)=[D_i]\text{ for }i=1,\dots,m\big\}.\]
\end{dfn}

\begin{rem}\label{rem:duality}
    Let $X$ be a $\Q$-factorial projective surface and let $\sigma\in\Cr(X)$. Since $\sigma$ acts linearly on $N_1(X)_\R$, it preserves not only the cone $\Eff(X)$, but also its closure $\overline{\Eff}(X)$. Moreover, since $\sigma$ preserves the intersection form, it also preserves the dual cone $\Nef(X)$. Then, for topological reasons, $\sigma$ preserves the interior of $\Nef(X)$, i.e.\ the ample cone.
\end{rem}

We will need the following lemma.

\begin{lem}\label{lem:sigmanodal}
    Let $X$ be a smooth projective surface and let $B\geq0$ be an $\R$-divisor on $X$ such that ${-}(K_X+B)$ is nef. Then every $\sigma\in\Cr(X,B)$ preserves $\Delta_{X,B}^{\nod}$.
\end{lem}

\begin{proof}
    Fix $\sigma\in\Cr(X,B)$. Let $C$ be a $({-}2)$-curve on $X$ such that $[C]\in \Delta_{X,B}^{\nod}$. Then $\sigma\big([C]\big)\in\Eff(X)$, hence there exist positive integers $m_1,\dots,m_k$ and irreducible curves $C_1,\dots,C_k$ on $X$ such that $\sigma\big([C]\big)=\big[\sum_{i=1}^k m_i C_i\big]$.  Therefore, $[C]=\sum_{i=1}^k m_i \sigma^{-1} \big([C_i]\big)$, and note that $\sigma^{-1} \big([C_i]\big)\in\Eff(X)$ for each $i=1,\dots,k$. Hence, for each $i$ there exists an effective integral divisor $D_i$ whose numerical class is $\sigma^{-1} \big([C_i]\big)$, and thus, $[C]=\sum_{i=1}^k m_i [D_i]$. By Lemma \ref{lem:negativeexceptional} we conclude that $C=\sum_{i=1}^k m_iD_i$, and therefore, $k=1$ and $m_1=1$. In particular, $\sigma\big([C]\big)=[C_1]$.

    Assume first that $C_1$ is a component of $B$. Then, since $\sigma^{-1}$ fixes the numerical class of each component of $B$, we have $[C]=\sigma^{-1}\big([C_1]\big)=[C_1]$, hence $C=C_1$ by Lemma \ref{lem:negativeexceptional}. This is a contradiction as $C$ is not a component of $B$. Therefore, $C_1$ is not a component of $B$. Since $[C_1]^2=\sigma\big([C]\big)^2=[C]^2={-}2$, we conclude that $[C_1]\in \Delta_{X,B}^{\nod}$ by Lemma \ref{lem:negativeselfintersection}, as desired.
\end{proof}

\subsection{Actions of groups on cones}

In this subsection we follow \cite{Loo14,LZ25}.

Let $\Lambda$ be a free $\Z$-module of finite rank and let $V\coloneqq \Lambda\otimes_\Z \R$. In this paper, a \emph{convex cone} in $V$ is a convex subset $\mathcal C\subseteq V$ such that $\lambda v\in\mathcal C$ for every $\lambda\geq0$ and $v\in\mathcal C$. It is \emph{strictly convex} if its closure $\overline{\mathcal C}$ contains no lines. It is \emph{polyhedral}, respectively \emph{rational polyhedral}, if it is spanned by finitely many elements of $V$, respectively of $\Lambda$.

For a convex cone $\mathcal C\subseteq V$ we denote by $\operatorname{Int}(\mathcal C)$ the interior of $\mathcal C$, and we define $\mathcal C^+$ as the convex hull of $\overline{\mathcal C}\cap \Lambda$. We say that a subgroup of $G\subseteq\GL(V)$ \emph{preserves $\mathcal C$} if $g(\mathcal C)=\mathcal C$ for each $g\in G$; then we also say that \emph{$G$ acts on $\mathcal C$}, and that we have an \emph{action of $G$ on $\mathcal C$}. Moreover, if $\Pi\subseteq\mathcal C$, then we denote $G\cdot\Pi\coloneqq\bigcup_{g\in G}g(\Pi)$.

\begin{dfn}
Let $\mathcal C\subseteq V$ be a strictly convex cone such that $\operatorname{Int}(\mathcal C)\neq\emptyset$. Let $G$ be a subgroup of $\GL(V)$ which fixes $\Lambda$ and preserves $\mathcal C$.
\begin{enumerate}[\normalfont (a)]
\item The action of $G$ on $\mathcal C$ is \emph{of polyhedral type} if there exists a polyhedral cone $\Pi \subseteq\mathcal C^+$ such that $\operatorname{Int}(\mathcal C)\subseteq G \cdot \Pi$.
\item The action of $G$ on $\mathcal C^+$ has a \emph{rational polyhedral fundamental domain} if there exists a rational polyhedral cone $\Pi \subseteq \mathcal C^+$ such that $G \cdot \Pi = \mathcal C^+$, and if $\operatorname{Int}\big(g(\Pi)\big) \cap \operatorname{Int}(\Pi)\neq\emptyset$ for some $g \in G$, then $g$ is the identity in $G$.
\end{enumerate} 
\end{dfn}

It is a remarkable fact from \cite{Loo14} that the two conditions in the previous definition are equivalent. For the proof of the following theorem, see \cite[Proposition 4.1, Application 4.14 and Corollary 4.15]{Loo14} as well as \cite[Lemma 3.5]{LZ25}.

\begin{thm}\label{thm:Looijenga}
Let $\Lambda$ be a free $\Z$-module of finite rank and let $V\coloneqq \Lambda\otimes_\Z \R$. Let $\mathcal C\subseteq V$ be a strictly convex cone such that $\operatorname{Int}(\mathcal C)\neq\emptyset$. Let $G$ be a subgroup of $\GL(V)$ which fixes $\Lambda$ and preserves $\mathcal C$. Assume that the action of $G$ on $\mathcal C$ is of polyhedral type.  Then $G \cdot \Pi = \mathcal C^+$, and there exists a rational polyhedral fundamental domain for the action of $G$ on $\mathcal C^+$.
\end{thm}

\section{Cones of divisors on surface Calabi--Yau generalised pairs}\label{sec:CY2}

In this section we prove several results on surface Calabi--Yau generalised pairs which will be crucial in the proof of Theorem \ref{thm:main1} and in \cite{LSX26}.

The following is the main result of this section.

\begin{thm}\label{thm:nefandeffectivecones}
    Let $(X,B+M)$ be a klt Calabi--Yau generalised pair of dimension $2$. Then:
    \begin{enumerate}[\normalfont (a)]
    \item $\Nef(X)^e\subseteq\Nef(X)^+$,
    \item if $M\not\equiv0$, then
    $$ \Eff(X)=\overline{\Eff}(X)\quad\text{and}\quad \Nef(X) = \Nef^e(X)=\Nef^+(X). $$
    \end{enumerate}
\end{thm}

We will prove Theorem \ref{thm:nefandeffectivecones} at the end of the section. We first need several preparatory results.

We start with the following theorem valid in every dimension, whose proof is analogous to that of \cite[Theorem 2.15]{LOP20}.

\begin{thm}\label{thm:nef+}
    Let $(X,B+M)$ be a klt Calabi--Yau generalised pair such that $M$ is a positive linear combination of nef $\Q$-divisors. Then
$$\Nef(X)^e\subseteq\Nef(X)^+.$$
\end{thm}

\begin{proof}
Let $D$ be an $\R$-divisor on $X$ whose numerical class is in $\Nef(X)^e$. There exists a real number $0<\varepsilon\ll1$ such that $\big(X,(B+\varepsilon D)+M\big)$ is a klt generalised pair. Then by \cite[Proposition 2.6]{HanLiu20} there exist nef divisors $N_1,\dots,N_\ell$ on $X$ and positive real numbers $r_1,\dots,r_\ell$ such that $K_X+(B+\varepsilon D)+M=r_1 N_1+\dots+r_\ell N_\ell$, and note that $K_X+(B+\varepsilon D)+M\equiv\varepsilon D$. This says that the numerical class of the divisor $\varepsilon D$ belongs to $\Nef(X)^+$, which implies the lemma.
\end{proof}

We will often need the following lemma.

\begin{lem}\label{lem:prepared}
    Let $(X,B+M)$ be a klt Calabi--Yau generalised pair of dimension $2$, and assume that there exists an $\R$-Cartier $\R$-divisor $D\geq0$ such that $D\equiv{-}K_X$. Let $D=N+P$ be the Zariski decomposition of $D$. Then $(X,N+P)$ is a klt Calabi--Yau generalised pair.
    
    Moreover, if $D$ is a Weil divisor, then $\Supp N\subseteq \Supp P=\Supp D$.
\end{lem}

\begin{proof}
The divisors $M$ and $P$ are nef. Thus, since
$$ M\equiv{-}K_X-B\equiv D-B, $$
we have we have $B\geq N$ by Lemma \ref{lem:Zariski}(c). Thus $(X,N)$ is a klt pair since $(X,B)$ is a klt pair, and since clearly $K_X+N+P\equiv0$, the couple $(X,N+P)$ is a klt Calabi--Yau generalised pair. This shows the first statement.

    For the second statement of the lemma, note first that the pair $(X,N)$ is klt by the paragraph above. In particular, for each component $\Gamma$ of $N$ we have $\mult_\Gamma N<1$. Since $\Supp N\subseteq\Supp D$ and $D$ is an integral divisor, we conclude that $\mult_\Gamma P=\mult_\Gamma D-\mult_\Gamma N>0$, which implies that $\Supp N\subseteq \Supp P$. Moreover, this gives that $\Supp D=\Supp(N+P)=\Supp P$, which completes the proof of the lemma.
\end{proof}

The next result, which rests heavily on the results from \cite{HanLiu20}, is instrumental for our paper.

\begin{lem}\label{lem:nonvanishing}
    Let $(X,B+M)$ be a log canonical Calabi--Yau generalised pair of dimension $2$. Then the following statements hold.
	\begin{enumerate}[\normalfont (a)]
	\item The divisor $M$ is num-effective.
	\item We have $\kappa(X,{-}K_X)\geq0$.
	\item There exist $\Q$-divisors $N$ and $P$ on $X$ such that $(X,N+P)$ is a log canonical Calabi--Yau generalised pair. Moreover, if $(X,B+M)$ is klt, then $(X,N+P)$ is klt.
	\item Assume that $(X,B+M)$ is klt and that $M\not\equiv0$. Let $L$ be a nef $\R$-divisor on $X$ which is not num-semiample. Then there exists a positive real number $\alpha$ such that $L\equiv \alpha M$.
	\end{enumerate}	    
   	\end{lem}

\begin{proof}
We first show (a) and (b). Since $M\equiv{-}(K_X+B)$, $M$ is nef and $(X,B)$ is log canonical, we conclude by \cite[Corollary 1.6]{HanLiu20} that there exists an $\R$-divisor $D\geq0$ such that ${-}(K_X+B)\sim_\R D$, which gives (a).

By the previous paragraph we have ${-}K_X\sim_\R B+D\geq0$, and therefore $\kappa(X,{-}K_X)\geq0$ by \cite[Proposition 2.5.9]{Fuj17}, which gives (b).

Next we show (c). By (b) there exists a $\Q$-divisor $G\geq0$ such that $G\sim_\Q{-}K_X$. Let $G=N+P$ be the Zariski decomposition of $G$. Then $(X,N+P)$ is the desired klt Calabi--Yau generalised pair by Lemma \ref{lem:prepared}.

Now we prove (d). The couple $\big(X,B+(M+L)\big)$ is a klt generalised pair, and the divisor $K_X+B+(M+L)\equiv L$ is nef. Thus, by \cite[Theorem 3.13]{HanLiu20} we infer that there exists $t\in[0,1]$ such that
    \begin{equation}\label{eq:1}
        -M \equiv K_X+B \equiv -t(M+L).
    \end{equation}
    Since $L$ is not num-semiample by our assumption, we have $t\neq 1$.
    Now if $t\neq 0$, we obtain $L\equiv \frac{1-t}{t}M$, and we set $\alpha\coloneqq \frac{1-t}{t}$; note that $\alpha$ must then be positive since $L$ and $M$ are both nef. If $t=0$, it follows from \eqref{eq:1} that $M\equiv 0$, a contradiction.
\end{proof}

Now we can prove Theorem \ref{thm:nefandeffectivecones}.

\begin{proof}[Proof of Theorem \ref{thm:nefandeffectivecones}]
For (a), note that by Lemma \ref{lem:nonvanishing}(c) we may assume that $B$ and $M$ are $\Q$-divisors, and we conclude by Theorem \ref{thm:nef+}.

In the remainder of the proof we prove (b). We first show that $\Nef(X) = \Nef^e(X)$. Indeed, let $D$ be a nef $\R$-divisor on $X$. If $D$ is num-semiample, then clearly $[D]\in\Nef^e(X)$. If $D$ is not num-semiample, then $[D]\in\Nef^e(X)$ by Lemma \ref{lem:nonvanishing}(a)(d).

Since $\Nef(X)^e\subseteq\Nef(X)^+$ by (a), and clearly $\Nef(X)^+\subseteq\Nef(X)$, we conclude that $\Nef(X) = \Nef^e(X)=\Nef^+(X)$ by the previous paragraph.

Finally, let $D$ be a pseudoeffective $\R$-divisor on $X$, and let $D=N+P$ be the Zariski decomposition of $D$. Then $N\geq0$ and $P$ is num-effective by the discussion above. Therefore, $[D]\in\Eff(X)$, which shows that $\Eff(X)=\overline{\Eff}(X)$.
\end{proof}

We finish this section with the following lemma which we will need in the proof of Theorem \ref{thm:liftingoneblowup}.

\begin{lem}\label{lem:pushgeneralised}
    Let $(X,B+M)$ be a klt Calabi--Yau generalised pair, and let $f\colon X\to X'$ be a birational morphism to a $\Q$-factorial surface $X'$. Denote $B'\coloneqq f_*B$ and $M'\coloneqq f_*M$. Then $(X',B'+M')$ is a klt Calabi--Yau generalised pair. If $M\not\equiv0$, then $M'\not\equiv0$.
\end{lem}

\begin{proof}
By the first paragraph of the proof of Lemma \ref{lem:nonvanishing} there exists an $\R$-divisor $D\geq0$ such that ${-}(K_X+B)\sim_\R D$. By replacing $M$ by $D$, and $M'$ by $f_*D$, we may assume that
\begin{equation}\label{eq:inequalities}
M\geq0\quad\text{and}\quad M'\geq0,
\end{equation}
and
\begin{equation}\label{eq:equalities}
K_X+B+M\sim_\R0\quad\text{and}\quad K_{X'}+B'+M'\sim_\R0.
\end{equation}
The $\R$-divisor $M'$ is nef by Lemma \ref{lem:pushforward}(d), hence by Lemma \ref{lem:pushforward}(f) there exists an $f$-exceptional divisor $E\geq0$ such that
\begin{equation}\label{eq:M+E}
f^*M'=M+E.
\end{equation}
Now \eqref{eq:equalities} and \eqref{eq:M+E} yield
$$K_X+B\sim_\R f^*(K_{X'}+B')+E,$$
from which is easily follows that the pair $(X',B')$ is klt.

For the last statement of the lemma, assume that $M'\equiv0$. Since $M'\geq0$ by \eqref{eq:inequalities}, we conclude that $M'=0$, and then \eqref{eq:M+E} yields $M+E\sim_\R0$. Since $M\geq0$ by \eqref{eq:inequalities}, this is possible only if $M=E=0$, as desired.
\end{proof}

\section{Non-rational surface Calabi--Yau generalised pairs}

In this section we prove a characterisation of surface Calabi--Yau generalised pairs which are not rational. It will be essential in the proof of Theorem \ref{thm:main1}. The results of this section generalise \cite[Lemma 1.4]{AM04}, and are interesting in their own right.

\begin{thm}\label{thm:nonrational}
Let $X$ be a smooth projective non-rational surface. Let $B\geq0$ be an $\R$-divisor whose coefficients are smaller than $1$, and let $M\geq0$ be a nef $\R$-divisor. Assume that $K_X+B+M\equiv0$ and $B+M\neq0$. Then:
\begin{enumerate}[\normalfont (a)]
\item $X$ is a ruled surface over a smooth elliptic curve,
\item ${-}K_X$ is nef with $K_X^2=0$,
\item all the components of $B+M$ are smooth disjoint curves, and each of them has self-intersection $0$,
\item the cone $\Nef(X)$ is rational polyhedral, and the classes on the two boundary rays on $\Nef(X)$ are not big.
\end{enumerate}
\end{thm}

\begin{proof}
The logic of the proof is similar to, but somewhat more involved than that of \cite[Lemma 1.4]{AM04}.

Let $f\colon X\to X'$ be a $K_X$-MMP. Since $K_X\equiv {-}(B+M)$ is not pseudoeffective and $X$ is not rational, we conclude that $X'$ is a ruled surface over a smooth curve $T$ of genus $g\geq1$. Let $\pi\colon X'\to T$ be the corresponding morphism. Set $B'\coloneqq f_*B$ and $M'\coloneqq f_*M$, and let $C_1,\dots,C_\ell$ be all the components of $B'+M'$.

We will prove the lemma in several steps.

\medskip

\emph{Step 1.}
In this step we prove:
\begin{enumerate}[\normalfont (i)]
\item $C_i\cdot C_j=0$ for all $i$ and $j$,
\item $K_{X'}\cdot C_i=0$ for all $i$,
\item ${-}K_{X'}$ is nef with $K_{X'}^2=0$,
\item the cone $\Nef(X')$ is rational polyhedral, and the classes on the two boundary rays on $\Nef(X')$ are not big,
\item $g=1$, hence $T$ is a smooth elliptic curve.
\end{enumerate} 

Note first that $K_{X'}+B'+M'\equiv0$ by Lemma \ref{lem:pushforward}(b), hence it is easy to see that (ii) and (iii) follow from (i) and from Kleiman's criterion for nefness. Also, (iv) follows from (iii), since then the cone $\Nef(X')$ is spanned by the numerical classes of ${-}K_{X'}$ and of a fibre of $\pi$. Thus, it suffices to show (i) and (v).

To that end, let $D$ be any curve on $X'$, and assume that $D^2<0$. Since the divisor ${-}(K_{X'}+B')\equiv M'$ is nef by Lemma \ref{lem:pushforward}(d), we have $(K_{X'}+B')\cdot D\leq0$, hence $D\simeq\mathbb{P}^1$ by \cite[Lemma 1.2]{AM04}. Since $D^2<0$, the curve $D$ cannot be a fibre of $\pi$. But this means that $D$ dominates $T$, hence $T$ is also a rational curve, a contradiction since $g\geq1$.

Therefore, each curve on $X'$ has a non-negative self-intersection. In particular, we have $(B'+M')^2\geq0$, with equality if and only if $C_i\cdot C_j=0$ for all $i$ and $j$. Now, since $g\geq1$ and ${-}K_{X'}\equiv B'+M'$, by \cite[Corollary V.2.11]{Har77} we have
$$0\geq 8(1-g)=K_{X'}^2=(B'+M')^2\geq0,$$
which immediately gives (v) and proves (i) by the previous sentence.

\medskip

\emph{Step 2.}
Fix an index $i\in\{1,\dots,\ell\}$. In this step we prove that $C_i$ is smooth.

Let $C_i^\nu$ be the normalisation of $C_i$. As $X'$ is a ruled surface over $T$, the divisor ${-}K_{X'}$ is $\pi$-ample, hence $C_i$ is not a fibre of $\pi$ since $K_{X'}\cdot C_i=0$ by Step 1. Therefore, $C_i^\nu$ dominates $T$.

By Step 1 we have $(K_{X'}+C_i)\cdot C_i=0$, hence $p_a(C_i)=1$ by the adjunction formula. This implies that $p_a(C_i^\nu)\leq1$. If $p_a(C_i^\nu)=0$, then $C_i^\nu$ is a smooth rational curve which dominates $T$ by the previous paragraph, which is a contradiction since $T$ is a smooth elliptic curve by Step 1. Therefore, $p_a(C_i^\nu)=1$, hence $C_i^\nu\simeq C_i$. This shows that $C_i$ is smooth.

\medskip

\emph{Step 3.}
In the remainder of the proof we will show that the map $f$ is an isomorphism. By Steps 1 and 2, this will immediately prove the lemma.

Arguing by contradiction, assume that $f$ is not an isomorphism, and let $h\colon X''\to X'$ be the last step in the $K_X$-MMP $f\colon X\to X'$. Let $B''$ and $M''$ be the strict transforms of $B$ and $M$ on $X''$.

We claim that
\begin{equation}\label{eq:pullbackM}
M''=h^*M'.
\end{equation}
Indeed, we have $(M')^2=0$ by part (i) from Step 1. In particular, since $M'$ is nef, we obtain that $\kappa(X',M')\leq1$. As $h_*M''=M'$, we have $\kappa(X'',M'')\leq\kappa(X',M')\leq1$ by Lemma \ref{lem:pushforward}(a), hence $M''$ is not big. Thus, $(M'')^2=0$ since $M''$ is nef by Lemma \ref{lem:pushforward}(d). We conclude by Lemma \ref{lem:pushforward}(g).

\medskip

\emph{Step 4.}
Finally, in this step we finish the proof that $f$ is an isomorphism. Since $K_X+B+M\equiv0$, we have $K_{X'}+B'+M'\equiv0$ and $K_{X''}+B''+M''\equiv0$ by Lemma \ref{lem:pushforward}(b). This, together with \eqref{eq:pullbackM}, gives 
\begin{equation}\label{eq:908o}
K_{X''}+B''\equiv h^*(K_{X'}+B').
\end{equation}
The map $h$ is the blowup of a point $x'\in X'$, and let $E''\subseteq X''$ be the $h$-exceptional curve. Set $ c\coloneqq \mult_{x'}B'\geq0$, and let $\widetilde B'$ be the strict transform of $B'$ on $X''$. Then $K_{X''}\sim h^*K_{X'}+E''$ and $h^*B'=\widetilde B'+cE''$. Since $B''\geq0$ by assumption, and since $B'=h_*B''$ by construction, there exists a real number $e\geq0$ such that $B''=\widetilde B'+eE''$. Combining all this with \eqref{eq:908o} gives
$$(e+1)E''\equiv cE''.$$
Intersecting both sides of this equation with $E''$ and using $(E'')^2={-}1$ gives $e+1=c$. As $e\geq0$, we conclude that
\begin{equation}\label{eq:contra}
c\geq1.
\end{equation}
In particular, $x'\in\Supp B'$. Since the components of $B'$ are disjoint by Step 1, the point $x'$ belongs to a unique component $C'$ of $B'$. As $C'$ is smooth by Step 2, this yields $ c=\mult_{C'}B' $. But all the coefficients of $B'$ are smaller than $1$, since all the coefficients of $B$ are smaller than $1$ by assumption. This implies $c<1$, which contradicts \eqref{eq:contra} and completes the proof.
\end{proof}

For singular surfaces, we have the following corollary.

\begin{cor}\label{cor:nonrational}
Let $(X,B+M)$ be a klt Calabi--Yau generalised pair of dimension $2$ such that $X$ is not rational, and assume that $B\neq0$ or that $M\not\equiv0$. Then $X$ is a ruled surface over a smooth elliptic curve and the divisor ${-}K_X$ is nef with $K_X^2=0$. Moreover, the cone $\Nef(X)$ is rational polyhedral, and the classes on the two boundary rays on $\Nef(X)$ are not big.
\end{cor}

\begin{proof}
Let $f\colon (\widetilde X,\widetilde B)\to(X,B)$ be the terminalisation of $(X,B)$ as in Theorem \ref{thm:terminalisation}. Then $\widetilde X$ is smooth. Since $f_*\widetilde B=B$, we have $\widetilde B\neq0$ if $B\neq0$.

The pair $(\widetilde X,\widetilde B+f^*M)$ is a klt Calabi--Yau generalised pair, hence by Lemma \ref{lem:nonvanishing}(a) there exists an $\R$-divisor $\widetilde M\geq0$ on $\widetilde X$ such that $\widetilde M\equiv f^*M$. Note that if $M\not\equiv 0$, then $\widetilde M\neq0$ by Lemma \ref{lem:pushforward}(b).

Therefore, by Theorem \ref{thm:nonrational} we have that $\widetilde X$ is a ruled surface over a smooth elliptic curve, the divisor ${-}K_{\widetilde X}$ is nef with $K_{\widetilde X}^2=0$, and the cone $\Nef(\widetilde X)$ is rational polyhedral such that the classes on the two boundary rays on $\Nef(\widetilde X)$ are not big.

It thus suffices to show that $f$ is an isomorphism. To that end, pick an ample divisor $A$ on $X$. Then the divisor $f^*A$ is nef and big on $\widetilde X$, hence its numerical class does not lie on the boundary of $\Nef(\widetilde X)$ by the previous paragraph. Thus, $f^*A$ is ample, hence $f$ is an isomorphism, as desired.
\end{proof}

\begin{rem}
We note that there exist surfaces as in Theorem \ref{thm:nonrational} and Corollary \ref{cor:nonrational}. Indeed, as in \cite[Example 1.7]{DPS94}, let $E$ be an elliptic curve, let $\mathcal E$ be a non-split vector bundle on $E$, and set $X\coloneqq \mathbb P(\mathcal E)$. If $C$ is the section of the projection $X\to E$ corresponding to the quotient $\mathcal E\to \OO_E$, then $C$ is a nef divisor on $X$ such that $K_X\sim{-}2C$.
\end{rem}

\section{Cremona isometries}\label{sec:Cremona}

In this section we prove that in the particular situations of this paper, the group of Cremona isometries lifts or descends under a birational morphism of smooth surfaces. The following is the main result of this section.

\begin{thm}\label{thm:cremonaiso}
Let $f\colon X\to X'$ be a birational morphism between $\Q$-factorial projective surfaces. Let $D$ be an effective divisor on $X$, set $D'\coloneqq f_*D$, and assume that $\Exc(f)\subseteq\Supp D$.\footnote{Note that the exceptional locus of $f$ is purely divisorial by \cite[1.40]{Deb01}.} Then the following statements hold.
\begin{enumerate}[\normalfont (a)]
\item There exists a well-defined group monomorphism
    $$\theta\colon\Cr(X,D) \to \Cr(X',D'),\quad \sigma\mapsto f_*\circ \sigma \circ f^*.$$
\item Assume, additionally, that $X'$ and $X$ are smooth and rational, that $\overline{\Eff}(X)=\Eff(X)$, and that ${-}(K_X+D)$ is nef. If $D$ is connected or if $\Delta_{X,D}^{\nod}=\emptyset$,\footnote{Recall the definition of the set $\Delta_{X,D}^{\nod}$ from Section \ref{sec:prelim}.} then $\theta$ is an isomorphism, and the inverse map is given by
    $$\theta^{-1}\colon\Cr(X',D') \to \Cr(X,D),\quad \rho\mapsto f^*\circ\rho\circ f_*+\id-f^*\circ f_* , $$
where $\id$ is the identity on $N^1(X)_\R$.
\end{enumerate}
\end{thm}

This theorem is a consequence of the following two lemmas. In them, we carefully keep track of the assumptions which ensure that certain properties lift or descend.

\begin{lem}\label{lem:isometries}
Let $f\colon X\to X'$ be a birational morphism between $\Q$-factorial projective surfaces. Let $E_1,\dots,E_k$ be all the $f$-exceptional prime divisors on $X$ and set $ V\coloneqq \sum_{i=1}^k \R [E_i] \subseteq N^1(X)_\R $. Let $\sigma\in\GL\big( N^1(X)_{\R} \big)$ and assume that $\sigma(V)=V$. Set
$$ \sigma'\coloneqq f_*\circ \sigma \circ f^*. $$
Then the following statements hold.
\begin{enumerate}[\normalfont (a)]
\item We have $\sigma'\in\GL\big( N^1(X')_{\R} \big)$.
\item Let $x\in N^1(X)_{\R}$ be such that $\sigma(x)=x$. Then $\sigma'(f_*x)=f_*x$.
\item If $\sigma\big([K_X]\big)=[K_X]$, then $\sigma'\big([K_{X'}]\big)=[K_{X'}]$.
\item If $\sigma$ preserves the cone $\Eff(X)$, then $\sigma'$ preserves the cone $\Eff(X')$.
\end{enumerate}
Assume, furthermore, that $\sigma(v)=v$ for each $v\in V$, and that $\sigma$ preserves the intersection product on $ N^1(X)_{\R} $. Then the following statements hold.
\begin{enumerate}[{\normalfont (e)}]
\item We have $\sigma\big(f^*N^1(X')_\R\big)\subseteq f^*N^1(X')_\R$.
\item[\normalfont{(f)}] If $\sigma$ preserves the cone $\Nef(X)$, then $\sigma\big(f^*\Nef(X')\big)\subseteq f^*\Nef(X')$.
\item[\normalfont{(g)}] The operator $\sigma'$ preserves the intersection product on $ N^1(X')_{\R} $.
\item[\normalfont{(h)}] If $\sigma$ preserves the lattice $N^1_W(X)$, then $\sigma'$ preserves the lattice $N^1_W(X')$.
\item[\normalfont{(i)}] Set $E\coloneqq E_1+\dots+E_k$. Then the group homomorphism
    $$\theta\colon\Cr(X,E) \to \GL\big( N^1(X')_{\R} \big),\quad \sigma\mapsto f_*\circ \sigma \circ f^*$$
	is well defined and injective. Moreover, for any effective divisor $D$ on $X$ we have $\theta\big(\Cr(X,D+E)\big)\subseteq\Cr(X',f_*D)$.
\end{enumerate}
\end{lem}

\begin{proof}
Clearly $\sigma'$ is an endomorphism of $N^1(X')_\R$, hence for (a) it suffices to show that it is an injective map. To that end, let $x'\in N^1(X')_\R$ be such that $\sigma'(x')=0$. In other words, $ f_*\big(\sigma(f^*x')\big)=0 $. By Lemma \ref{lem:pushforward}(c) this implies that $\sigma(f^*x')\in V$. By the assumption of the lemma, this implies $f^*x'\in V$, and pushing this forward by $f_*$ gives $x'=0$. Thus, $\sigma'$ is injective.

For (b), note that $f_*x=f_*\big(\sigma(x)\big)$ as $\sigma(x)=x$. Also, $f^*f_*x-x\in V$ by Lemma \ref{lem:pushforward}(c), hence $\sigma(f^*f_*x-x)\in V$ as $\sigma(V)=V$. Therefore,
\begin{align*}
\sigma'(f_*x)-f_*x&=f_*\big(\sigma(f^*f_*x)\big)-f_*\big(\sigma(x)\big)\\
&=f_*\big(\sigma(f^*f_*x-x)\big)\in f_*(V)=\{0\}.
\end{align*}
This gives (b), and part (c) follows from (b) since $f_*[K_X]=[K_{X'}]$.

For (d), let $x'\in\Eff(X')$. Then $f^*x'\in\Eff(X)$, thus $\sigma(f^*x')\in\Eff(X)$ by assumption. Therefore, Lemma \ref{lem:pushforward}(b) gives
$$ \sigma'(x')=f_*\big(\sigma(f^*x')\big)\in f_*\Eff(X)\subseteq\Eff(X'), $$
which was to be shown.

We now show (e) and (f). Let $x'\in N^1(X')_\R$. By Lemma \ref{lem:pushforward}(c) there exists $y'\in N^1(X')_\R$ and $e\in V$ such that $\sigma(f^*x')=f^*y'+e$. By assumption we have $\sigma(e)=e$, and $\sigma$ preserves the intersection product on $ N^1(X)_{\R} $. Therefore, by the projection formula and by Lemma \ref{lem:pushforward}(e) we obtain
$$0=f^*x'\cdot e=\sigma(f^*x')\cdot\sigma(e)=(f^*y'+e)\cdot e=e^2\leq0.$$
This shows that $e^2=0$, and hence $e=0$ again by Lemma \ref{lem:pushforward}(e). Thus,
\begin{equation}\label{eq:sigmanef}
\sigma(f^*x')=f^*y',
\end{equation}
which proves (e). If, moreover, $x'$ is nef and if $\sigma$ preserves $\Nef(X)$, then $\sigma(f^*x')$ is also nef. Therefore, $f^*y'\in\Nef(X)$ by \eqref{eq:sigmanef}, and consequently, $y'\in\Nef(X')$, which proves (f).

Next we show (g). Fix $x',y'\in N^1(X')_\R$. By (e) there exist $z',w'\in N^1(X')_\R$ such that
$$\sigma(f^*x')=f^*z' \quad\text{and}\quad\sigma(f^*y')=f^*w'.$$
Therefore, by these relations, by the projection formula and since $\sigma$ preserves the intersection product on $ N^1(X)_{\R} $, we obtain
\begin{align*}
x'\cdot y' &= f^*x'\cdot f^*y' = \sigma(f^*x') \cdot \sigma(f^*y') = f^*z'\cdot f^*w' = z'\cdot w'\\
&=f_*f^*z'\cdot f_*f^*w' = f_*\big(\sigma(f^*x')\big) \cdot f_*\big(\sigma(f^*y')\big)=\sigma'(x')\cdot\sigma'(y'),
\end{align*}
as desired.

Next we show (h). Let $G'$ be a Weil divisor on $X'$, and let $G$ be the strict transform of $G'$ on $X$. Since $X'$ is $\Q$-factorial, we have $f^*G'=G+F$, where $F$ is an exceptional $\Q$-divisor on $X$. By assumption we have $\sigma\big([F]\big)=[F]$, and there exists a Weil divisor $G_1$ on $X$ such that $\sigma\big([G]\big)=[G_1]$. Therefore,
\begin{align*}
\sigma'\big([G']\big)&=f_*\big(\sigma\big(f^*[G']\big)\big)=f_*\big(\sigma\big([G]+[F]\big)\big)\\
&=f_*\big([G_1]+[F]\big)=f_*[G_1]=[f_*G_1]\in N^1_W(X'),
\end{align*}
as desired.

Finally, we prove (i). The map $\theta$ is well defined by (a). To show that $\theta$ is injective, let $\sigma\in\Cr(X,E)$ be such that $\theta(\sigma)$ is the identity on $N^1(X')_\R$. Then for any $x'\in N^1(X')_\R$ we have
$$ x'=\theta(\sigma)(x')=f_*\big(\sigma(f^*x')\big). $$
By (e) there exists $y'\in N^1(X')_\R$ such that $\sigma(f^*x')=f^*y'$, and hence the equation above gives
$$x'=f_*f^*y'=y'.$$
Thus, $\sigma(f^*x')=f^*x'$ for all $x'\in N^1(X')_\R$, hence $\sigma$ is the identity on the subspace $f^*N^1(X')_\R\subseteq N^1(X)_\R$. Since $\sigma$ is the identity on the subspace $V\subseteq N^1(X)_\R$ by assumption, we conclude that $\sigma$ is the identity on $N^1(X)_\R$ by Lemma \ref{lem:pushforward}(c). Thus, $\theta$ is injective. The last statement of (i) follows from (b), (c), (d), (g) and (h).
\end{proof}

For the following lemma, recall the definitions of sets $\mathcal C_X^+$, $\mathcal M_X$ and $\Delta_{X,D}^{\nod}$ from Section \ref{sec:prelim}.

\begin{lem}\label{lem:isometriesdownup}
Let $f\colon X\to X'$ be a birational morphism between $\Q$-factorial projective surfaces. Let $\rho'\in\GL\big( N^1(X')_{\R} \big)$, and set
$$ \rho\coloneqq f^*\circ\rho'\circ f_*+\id-f^*\circ f_* , $$
where $\id$ is the identity on $N^1(X)_\R$. Then the following statements hold.
\begin{enumerate}[\normalfont (a)]
\item We have $\rho\in\GL\big( N^1(X)_{\R} \big)$.
\item Let $x\in N^1(X)_{\R}$. If $\rho'(f_*x)=f_*x$, then $\rho(x)=x$.
\item If $\rho'\big([K_{X'}]\big)=[K_{X'}]$, then $\rho\big([K_X]\big)=[K_X]$.
\item If $X$ and $X'$ are smooth, and if $\rho'$ preserves the lattice $N^1(X')$, then $\rho$ preserves the lattice $N^1(X)$.
\end{enumerate}
Assume, furthermore, that $\rho'$ preserves the intersection product on $ N^1(X')_{\R} $. Then the following statements hold.
\begin{enumerate}[{\normalfont (e)}]
\item The operator $\rho$ preserves the intersection product on $ N^1(X)_{\R} $.
\end{enumerate}
Assume, furthermore, that there exists an ample class $h'\in N^1(X')$ such that $\rho'(h')$ is ample.
\begin{enumerate}[{\normalfont (f)}]
\item[\normalfont{(f)}] There exists an ample class $h\in N^1(X)$ such that $\rho(h)$ is ample.
\item[\normalfont{(g)}] If $X$ is smooth, we have $\rho(\mathcal C_X^+)=\mathcal C_X^+$.
\end{enumerate}
Assume, furthermore, that $X$ and $X'$ are smooth and that ${-}K_X$ is pseudoeffective.
\begin{enumerate}[{\normalfont (h)}]
\item If $X$ is rational, if $\rho'\big([K_{X'}]\big)=[K_{X'}]$, and if $\rho'$ preserves the lattice $N^1(X')$, then $\rho(\mathcal M_X)=\mathcal M_X$.
\item[\normalfont{(i)}] Let $D\geq0$ be a connected $\R$-divisor on $X$ such that ${-}(K_X+D)$ is nef. Assume that $\Exc(f)\subseteq\Supp D$ and that $\rho'\in \Cr(X',D')$.
Then $\rho(\Delta_{X,D}^{\nod})=\Delta_{X,D}^{\nod}$.
\end{enumerate}
\end{lem}

\begin{proof}
Clearly $\rho$ is an endomorphism of $N^1(X)_\R$, hence for (a) it suffices to show that it is an injective map. To that end, let $x\in N^1(X)_\R$ be such that $\rho(x)=0$. Then $f_*\big(\rho(x)\big)=0$. Since $f_*\circ\rho=\rho'\circ f_*$ by the definition of $\rho$, we obtain $\rho'(f_*x)=0$, hence $f_*x=0$ as $\rho'$ is bijective. Thus,
$$ 0=\rho(x)=f^*\big(\rho'(f_*x)\big)+x-f^*f_*x=x. $$
Thus, $\rho$ is injective, which gives (a). Part (d) is obvious.

For (b), we easily calculate
$$ \rho(x) = f^* \big(\rho'(f_*x)\big) + x - f^*f_*x = f^*f_*x + x - f^*f_*x = x, $$
as desired. Part (c) follows from (b) since $f_*[K_X]=[K_{X'}]$.

For (e), fix $x,y\in N^1(X)_\R$. By Lemma \ref{lem:pushforward}(c) there exist $x',y'\in N^1(X')_\R$ as well as $e_x,e_y\in N^1(X)_\R$, which are numerical classes of $f$-exceptional $\R$-divisors on $X$, such that
$$x=f^*x'+e_x \quad\text{and}\quad y=f^*y'+e_y.$$
Then it is an easy calculation that
$$\rho(x)=f^*\big(\rho'(x')\big) + e_x \quad\text{and}\quad\rho(y)=f^*\big(\rho'(y')\big) + e_y.$$
Therefore, by the projection formula and since $\rho'$ preserves the intersection product on $ N^1(X')_{\R} $, we obtain
\begin{align*}
x\cdot y &= \big(f^*x' + e_x\big) \cdot \big(f^*y' + e_y\big) = f^*x' \cdot f^*y'+e_x \cdot e_y = x'\cdot y' +e_x \cdot e_y \\
&= \rho'(x')\cdot\rho'(y') +e_x \cdot e_y = f^*\big(\rho'(x')\big) \cdot f^*\big(\rho'(y')\big) +e_x \cdot e_y\\
&= \big(f^*\big(\rho'(x')\big) + e_x\big) \cdot \big(f^*\big(\rho'(y')\big) + e_y\big) = \rho(x) \cdot \rho(y),
\end{align*}
hence $\rho$ preserves the intersection product on $ N^1(X)_{\R} $.

Now we prove (f). By \cite[Proposition 1.7.10]{Laz04} and \cite[Lemma 2.62]{KM98} there exist $m\in\Z_{>0}$ and $e\in N^1(X)$, which is the numerical class of an $f$-exceptional divisor on $X$, such that the classes $ mf^* h' - e $ and $ mf^*\big(\rho'(h')\big) - e $ are ample. Set $h\coloneqq mf^* h' - e$. We easily calculate that $\rho(h)=mf^*\big(\rho'(h')\big) - e$, hence both $h$ and $\rho(h)$ are ample.

For (g), fix an ample class $h\in N^1(X)$ as in (f), and denote by $h_1$ the ample class $\rho(h)$. Then by (e) we have
    \begin{align*}
    \rho(\mathcal C_X^+) &= \big\{ \rho(x) \mid x\in N^1(X)_\R,\; x^2 >0,\, x\cdot h >0 \big\}\\
    &= \big\{ \rho(x) \mid x\in N^1(X)_\R,\, \rho(x)^2 >0,\, \rho(x)\cdot h_1 >0 \big\}\\
    &= \big\{ y\in N^1(X)_\R \mid y^2 >0,\, y\cdot h_1 >0 \big\}=\mathcal C_X^+,
    \end{align*}
   as desired.

For (h), fix an ample class $h\in N^1(X)$ as in (f), and denote by $h_1$ the ample class $\rho(h)$. We have $\rho\big([K_X]\big)=[K_{X}]$ by the assumptions of (h) and by (c), and $\rho$ preserves the lattice $N^1(X)$ by the assumptions of (h) and by (d). Therefore, by (e) and by using Lemma \ref{lem:rationalsurface-1curve}(b) twice, we obtain
    \begin{align*}
    \rho(\mathcal{M}_X) &= \big\{ \rho(x) \mid x\in N^1(X),\, x^2 = [K_X]\cdot x = {-}1,\, x\cdot h>0 \big\}\\
    &= \big\{ \rho(x) \mid x\in N^1(X),\, \rho(x)^2 = \rho([K_X])\cdot \rho(x) = {-}1,\, \rho(x)\cdot h_1>0 \big\}\\
    &= \big\{ y\in N^1(X) \mid y^2 = [K_X]\cdot y = {-}1,\ y\cdot h_1 >0 \big\}=\mathcal M_X,
    \end{align*}
    which was to be proved.
    
    Finally, we show (i). Set $D'\coloneqq f_*D$. Let $C$ be a curve on $X$ such that $C^2={-}2$ and $C\cap\Supp D=\emptyset$, and set $C'\coloneqq f_* C$. As $\Exc(f)\subseteq\Supp D$ by assumption, we conclude that $C'\cap\Supp D'=\emptyset$ and
\begin{equation}\label{eq:rho0}
C = f^* C'.
\end{equation}
Therefore, $(C')^2=C^2={-}2$, which implies $[C']\in\Delta_{X',D'}^{\nod}$. As the $\R$-divisor ${-}(K_{X'}+D')={-}f_*(K_X+D)$ is nef by Lemma \ref{lem:pushforward}(d), we have $\rho'\big([C']\big)\in \Delta_{X',D'}^{\nod}$ by Lemma \ref{lem:sigmanodal}. Hence, there exists a curve $C_1'$ on $X'$ such that $(C_1')^2={-}2$ and $C_1'\cap\Supp D'=\emptyset$, such that 
\begin{equation}\label{eq:rho1}
\rho'\big([C']\big)=[C_1'].
\end{equation}
Now, $D$ is connected and $\Exc(f)\subseteq\Supp D$ by assumption. This implies that $f\big(\Exc(f)\big)\subseteq\Supp D'$, hence the curve $C_1\coloneqq f^*C_1'$ is irreducible, disjoint from $\Supp D$, and $(C_1)^2=(C_1')^2={-}2$. This shows that $[C_1]\in \Delta_{X,D}^{\nod}$ and thus, by \eqref{eq:rho0} and \eqref{eq:rho1} we have
\begin{align*}
\rho\big([C]\big) &= f^*\big(\rho'\big(f_*[C]\big)\big)+[C]-f^*f_*[C]\\
&=f^*\big(\rho'\big([C']\big)\big)=f^* [C_1']=[C_1]\in \Delta_{X,D}^{\nod}.
\end{align*}
Therefore, $\rho(\Delta_{X,D}^{\nod})\subseteq\Delta_{X,D}^{\nod}$. The converse inclusion follows by considering $\rho^{-1}$ instead of $\rho$. This finishes the proof.
\end{proof}

We can now prove Theorem \ref{thm:cremonaiso}.

\begin{proof}[Proof of Theorem \ref{thm:cremonaiso}]
Let $E_1,\dots,E_k$ be all the $f$-exceptional prime divisors on $X$. Since $E_i\subseteq\Supp D$ for each $i$ by assumption, we have $ \sigma\big([E_i]\big)=[E_i] $ for every $\sigma\in\Cr(X,D)$ and for all $i=1,\dots,k$. Then $\theta$ is well defined and it is a monomorphism by Lemma \ref{lem:isometries}(i). This gives (a).

In the remainder of the proof we show (b). Let $\rho'\in\Cr(X',D')$, and set
$$ \rho\coloneqq f^*\circ\rho'\circ f_*+\id-f^*\circ f_* , $$
where $\id$ is the identity on $N^1(X)_\R$. It is easy to see that $\theta(\rho)=\rho'$. We will show that $\rho\in\Cr(X,D)$, which will prove that $\theta$ is surjective and $\theta^{-1}(\rho')=\rho$.

By Lemma \ref{lem:isometriesdownup}(c)(d)(e) we have that $\rho$ fixes the class $[K_X]$, it preserves the lattice $N^1(X)$, and it preserves the intersection product on $N^1(X)_\R$. Since $D'=f_*D$, the operator $\rho$ also fixes the numerical classes of all the components of $D$ by Lemma \ref{lem:isometriesdownup}(b). To show that $\rho\in\Cr(X,D)$, it remains to show that $\rho$ preserves the cone $\Eff(X)$. Since $\overline{\Eff}(X)=\Eff(X)$ by assumptions of part (b), by Lemma \ref{lem:rational_Mori_cone} and by what we showed above, it suffices to prove that $\rho$ preserves the sets $\mathcal C_X^+$, $\mathcal M_X$ and $\Delta_{X,D}^{\nod}$.

To that end, $\rho$ preserves the sets $\mathcal C_X^+$ and $\mathcal M_X$ by Remark \ref{rem:duality} and Lemma \ref{lem:isometriesdownup}(g)(h). Thus, we are finished if $\Delta_{X,D}^{\nod}=\emptyset$. If $D$ is connected, then $\rho$ preserves $\Delta_{X,D}^{\nod}$ by Lemma \ref{lem:isometriesdownup}(i). This concludes the proof.
\end{proof}

\section{Lifting the Cone Conjecture}

In this section we show that in certain situations, which fit perfectly with the setup of this paper, a polyhedral action of a group of Cremona isometries can be lifted by a birational morphism. The following is the main result of the section.

\begin{thm}\label{thm:liftingoneblowup}
    Let $(X,B+M)$ be a klt Calabi--Yau generalised pair such that $X$ is a smooth rational surface and $M\not\equiv0$, and let $D\geq0$ be an integral divisor such that $D\equiv -K_X$. Let $f\colon X\to X'$ be a birational morphism to a smooth surface $X'$, set $D'\coloneqq f_*D$, and assume that $\Exc(f)\subseteq\Supp D$. Assume that $D$ is connected or that $\Delta_{X,D}^{\nod}=\emptyset$.\footnote{Recall the definition of the set $\Delta_{X,D}^{\nod}$ from Section \ref{sec:prelim}.}
    
    If there exists a rational polyhedral fundamental domain for the action of the group $\Cr(X',D')$ on $\Nef^e(X')$, then there exists a rational polyhedral fundamental domain for the action of the group $\Cr(X,D)$ on $\Nef^e(X)$.
\end{thm}

\begin{proof}
The proof is inspired by the proof of \cite[Theorem 24]{Xu24}.

    Let $D=N+P$ be the Zariski decomposition of $D$. Then $(X,N+P)$ is a klt Calabi--Yau generalised pair by Lemma \ref{lem:prepared}, and
    \begin{equation}\label{eq:supports}
    \Supp N \subseteq \Supp P = \Supp D.
    \end{equation}
    By Theorem \ref{thm:nefandeffectivecones}(b) we have
    \begin{equation}\label{eq:twoequalities1}
    \Nef(X)=\Nef^e(X)=\Nef^+(X)\quad\text{and}\quad \overline{\Eff}(X)=\Eff(X),
    \end{equation}
    and by Lemma \ref{lem:pushgeneralised} and by Theorem \ref{thm:nefandeffectivecones}(b) we have
    \begin{equation}\label{eq:twoequalities2}
    \Nef(X')=\Nef^e(X').
    \end{equation}
    
    \medskip
    
	\emph{Step 1.}
    Let $\Pi'$ be a rational polyhedral fundamental domain of $\Nef(X')$ for the action of $\Cr(X',D')$. Then by \eqref{eq:twoequalities2} there exist effective $\Q$-divisors $G_1',\dots, G_m'$ on $X'$ such that $\Pi'=\sum_{i=1}^m\R_{\geq0}[G_i']$. Let $F_1,\dots,F_s$ be all the prime components of the $\Q$-divisor $D+f^*G_1'+\dots+f^*G_m'$. Denote
    $$ V \coloneqq  \bigoplus_{i=1}^s [0,1]F_i\oplus[0,1]P $$
    and
    $$ \mathcal L(V)\coloneqq \{ \Delta+\mu P\in V\mid (X,\Delta + \mu P) \text{ is a log canonical generalised pair}\}.$$
    By \cite[Proposition 3.20]{HL22},\footnote{With notation from \cite[Proposition 3.20]{HL22}, we apply op.\ cit.\ to the set $T$ of \emph{all} extremal rays of the closure of the cone of effective curves on $X$. In that case, our set $\mathcal N$ is the image of the set $K_X+\mathcal N_T$ from op.\ cit.\ under the canonical map from $K_X+V$ to $N^1(X)_\R$.} the set
    $$\textstyle \mathcal{N} \coloneqq  \{ [K_X+\Delta+\mu P]\mid \Delta+\mu P\in\mathcal L(V),\, K_X+\Delta+\mu P\text{ is nef}\}\subseteq N^1(X)_\R $$
    is a rational polytope, hence the set
    $$ \mathcal C\coloneqq \R_{\geq0}\mathcal{N}\subseteq \Nef(X)$$
    is a rational polyhedral cone.
    
    \medskip
    
    \emph{Step 2.}
    Consider the linear map $f_*\colon N^1(X)_\R\to N^1(X')_\R$, and denote $\Pi\coloneqq (f_*)^{-1}(\Pi') \cap \Nef(X)\subseteq N^1(X)_\R$. We show in this step that
    \begin{equation}\label{eq:08}
    \Pi\subseteq\mathcal C.
    \end{equation}
    
    To that end, let $x \in \Pi$. Then $x'\coloneqq f_*x\in \Pi'$, hence there exist non-negative real numbers $g_1,\dots,g_m$ such that
    \begin{equation}\label{eq:8989}
    x' = \sum_{i=1}^m g_i [G_i'].
    \end{equation}
    Denote
    $$G\coloneqq \sum_{i=1}^m g_i f^* G_i', $$
    and note that $G$ is an effective $\R$-divisor on $X$. By Lemma \ref{lem:pushforward}(c) there exists an $f$-exceptional $\R$-divisor $E$ such that $x=f^*x'+[E]$, which together with \eqref{eq:8989} gives
    \begin{equation}\label{eq:8989a}
    x = [G+E].
    \end{equation}
    Now, since $\Supp E\subseteq\Supp D$ by assumption, and since $\Supp P=\Supp D$ by \eqref{eq:supports}, we may pick $0<\varepsilon\ll\delta\ll1$ such that
    $$\delta P+\varepsilon E\geq0,$$
    and such that
    $$\text{the pair }(X,N+\varepsilon G+\delta P+\varepsilon E)\text{ is klt.}$$
    Since $\Supp(N+\varepsilon G+\delta P+\varepsilon E)\subseteq\sum_{i=1}^sF_i$, we conclude that
    \begin{equation}\label{eq:08a}
    (N+\varepsilon G + \delta P + \varepsilon E ) + (1-\delta) P\in V.    
    \end{equation}
    Furthermore, since $K_X+P+N \equiv 0$, from \eqref{eq:8989a} we obtain
    \begin{align}\label{eq:coeff}
    \varepsilon x &= [K_X+N+P] + \varepsilon [G+E] \\
    &= \big[K_X+(N+\varepsilon G + \delta P + \varepsilon E ) + (1-\delta) P\big].\notag
    \end{align}
    Consequently, the fact that $x$ is nef, together with \eqref{eq:08a} and \eqref{eq:coeff}, gives
    $$\big[K_X+(N+\varepsilon G + \delta P + \varepsilon E ) + (1-\delta) P\big]\in \mathcal N,$$
    hence $x\in \mathcal C$ by \eqref{eq:coeff}. This shows \eqref{eq:08}.
    
    \medskip
    
    \emph{Step 3.}
    In this step we show that
    \begin{equation}\label{eq:08b}
	\Nef(X)\subseteq \Cr(X,D)\cdot \mathcal C.    
    \end{equation}
    This immediately implies the theorem: indeed, by Theorem \ref{thm:Looijenga}, the inclusion \eqref{eq:08b} yields that there exists a rational polyhedral fundamental domain for the action of $\Cr(X,D)$ on $\Nef^+(X)$, and we conclude by \eqref{eq:twoequalities1}.
    
    To prove \eqref{eq:08b}, let $y\in\Nef(X)$ and set $y'\coloneqq f_*y$. Then $y'\in\Nef^e(X')$ by Lemma \ref{lem:pushforward}(d) and by \eqref{eq:twoequalities2}. Since $\Pi'$ is a fundamental domain for the action of $\Cr(X',D')$ on $\Nef^e(X')$, there exists $\rho' \in \Cr(X',D')$ such that $\rho'(y')\in \Pi'$. Denote
	$$\rho\coloneqq f^*\circ\rho'\circ f_*+\id-f^*\circ f_*,$$
	where $\id$ is the identity on $N^1(X)_\R$. By Lemma \ref{lem:pushforward}(c) there exists a numerical class $e$ of an $f$-exceptional $\R$-divisor, such that $ y = f^*y' + e $, and it is easy to calculate that $ \rho(y)=f^*\big(\rho'(y')\big)+e $. Thus,
	\begin{equation}\label{eq:rhos}
	f_*\big(\rho(y)\big)=\rho'(y')\in\Pi'.	
	\end{equation}
	Moreover, by \eqref{eq:twoequalities1} and by Theorem \ref{thm:cremonaiso} we have $\rho\in\Cr(X,D)$. Thus, $\rho(y)\in\Nef(X)$ since $y\in\Nef(X)$. Therefore, $\rho(y)\in\Pi$ by the definition of $\Pi$ and by \eqref{eq:rhos}, hence $\rho(y)\in\mathcal C$ by \eqref{eq:08}. Consequently, $y\in\rho^{-1}(\mathcal C)$, which shows \eqref{eq:08b} and finishes the proof.
\end{proof}

\section{Descending the Cone Conjecture}

In this section we show that in certain situations, which fit perfectly with the setup of this paper, a polyhedral action of a group of Cremona isometries can descend by a birational morphism. The following is the main result of the section.

\begin{thm}\label{thm:descent}
    Let $(X,B+M)$ be a klt Calabi--Yau generalised pair with $M\not\equiv0$, and let $f\colon (X',B')\to (X,B)$ be the $\Q$-factorial terminalisation of $(X,B)$. Assume that there exists an integral divisor $D' \sim_\Q {-}K_{X'}$, and set $D\coloneqq f_*D'$. Assume that $\kappa(X',{-}K_{X'})=0$ and $\Delta_{X',D'}^{\nod}=\emptyset$.\footnote{Recall the definition of the set $\Delta_{X',D'}^{\nod}$ from Section \ref{sec:prelim}.}
    
    If there exists a rational polyhedral fundamental domain for the action of the group $\Cr(X',D')$ on $\Nef^e(X')$, then there exists a rational polyhedral fundamental domain for the action of the group $\Cr(X,D)$ on $\Nef^e(X)$.
\end{thm}

\begin{proof}
We prove the theorem in two steps. Step 2 is inspired by the proofs of \cite[Theorem 24]{Xu24} and \cite[Proposition 5.3]{GLSW26}.

\medskip

\emph{Step 1.}
In this step we show that every $f$-exceptional curve on $X'$ belongs to $\Supp D'$.

The couple $(X',B'+f^*M)$ is a klt Calabi--Yau generalised pair. Thus, by the first paragraph of the proof of Lemma \ref{lem:nonvanishing} there exists an $\R$-divisor $M'\geq0$ such that ${-}(K_{X'}+B')\sim_\R M'$. Since $D'\sim_\Q{-}K_{X'}\sim_\R B'+M'$ and since $\kappa(X',D')=\kappa(X',{-}K_{X'})=0$ by assumption, we have $D'=B'+M'$ by \cite[Lemma 2.6(b)]{Laz26}. In particular, $0\leq B'\leq D'$. Then, we may write
$$D'=\sum_{i=1}^k d_iD_i' \quad\text{and}\quad B'=\sum_{i=1}^k b_iD_i',$$
where for each $i$ we have $d_i\in\Z_{>0}$, $D_i'$ is a prime divisor, and $0\leq b_i<1$.

Note that $X'$ is smooth. Let $E'$ be an $f$-exceptional curve on $X'$ and assume, for contradiction, that $E'\nsubseteq\Supp D'$. We have $(E')^2<0$ by Lemma \ref{lem:pushforward}(e), and since $\Delta_{X',D'}^{\nod}=\emptyset$, the curve $E'$ is a $(-1)$-curve such that $D'\cdot E'=1$ by Lemma \ref{lem:negativeselfintersection}. Then, as $E'\nsubseteq\Supp D'$, we conclude that, without loss of generality,
\begin{equation}\label{eq:345}
d_1=1,\quad D_1'\cdot E'=1\quad\text{and}\quad D_i'\cdot E'=0\text{ for all }i\geq2.
\end{equation}
On the other hand, since $K_{X'}+B'\sim_\R f^*(K_X+B)$, we have $ (K_{X'}+B')\cdot E'=0 $. As $K_{X'}\cdot E'={-}1$, we conclude that $B'\cdot E'=1$. But by \eqref{eq:345} we have $B'\cdot E'=b_1<1$, a contradiction.

\medskip

\emph{Step 2.}
By Step 1 and by Lemma \ref{lem:isometries}(f), the group $\Cr(X',D')$ preserves $f^*\Nef(X)$. Since $\Nef(X) = \Nef^+(X)$ by Theorem \ref{thm:nefandeffectivecones}, we have that $f^*\Nef(X)$ is a face of the cone $\Nef^+(X')$ by the first paragraph of the proof of \cite[Proposition 5.3]{GLSW26}. As $\Nef^+(X')=\Nef^e(X')$ by Theorem \ref{thm:nefandeffectivecones}, and as the action of $\Cr(X',D')$ on $\Nef^e(X')$ is of polyhedral type by the assumptions of the theorem and by Theorem \ref{thm:Looijenga}, we conclude by \cite[Proposition 3.7]{GLSW26} that the action of $\Cr(X',D')$ on $f^*\Nef(X)$ is of polyhedral type. Thus, there exists a polyhedral cone $\Pi'\subseteq f^* \Nef(X)$ such that 
\begin{equation}\label{eq:909}
f^* \Amp(X)\subseteq \bigcup_{\sigma'\in\Cr(X',D')}\sigma'(\Pi').
\end{equation}
There exist classes $x_1,\dots,x_\ell\in \Nef(X)$ such that $\Pi'=\sum_{i=1}^\ell\R_{\geq0}f^*x_i$. If we denote
$$\Pi\coloneqq f_*\Pi'=\sum_{i=1}^\ell\R_{\geq0}x_i\subseteq \Nef(X),$$
then $\Pi$ is a polyhedral cone, and we clearly have $\Pi'=f^*\Pi$. This, together with \eqref{eq:909}, gives
\begin{align*}
\Amp(X)&=f_*f^*\Amp(X)\subseteq f_*\Big(\bigcup\nolimits_{\sigma'\in\Cr(X',D')}\sigma'(\Pi')\Big)\\
&= \bigcup\nolimits_{\sigma'\in\Cr(X',D')}(f_*\circ\sigma')(\Pi')=\bigcup\nolimits_{\sigma'\in\Cr(X',D')}(f_*\circ\sigma'\circ f^*)(\Pi).
\end{align*}
By Step 1 and by Theorem \ref{thm:cremonaiso} we have $f_*\circ \sigma' \circ f^*\in \Cr(X,D)$ for each $\sigma'\in\Cr(X',D')$, hence by the above we conclude that
$$ \Amp(X)\subseteq \bigcup\nolimits_{\sigma\in\Cr(X,D)}\sigma(\Pi).$$
Thus, the action of $\Cr(X,D)$ on $\Nef(X)$ is of polyhedral type. Since we have $\Nef(X) = \Nef^e(X)$ by Theorem \ref{thm:nefandeffectivecones}, there exists a rational polyhedral fundamental domain for the action of the group $\Cr(X,D)$ on $\Nef^e(X)$ by Theorem \ref{thm:Looijenga}.
\end{proof}

\section{Proofs of Theorems \ref{thm:main0} and \ref{thm:main1}}

In this section we finally prove Theorems \ref{thm:main0} and \ref{thm:main1}. We assume Theorem \ref{thm:main2}, which will be proved in the sequel to this paper \cite{LSX26}.

We will need the following lemma in the proof of Theorem \ref{thm:main1}.

\begin{lem}\label{lem:semiample}
    Let $(X,B+M)$ be a klt Calabi--Yau generalised pair of dimension $2$. Let $D\geq0$ be a $\Q$-Cartier divisor such that $D\sim_\Q{-}K_X$, and let $P$ be the positive part in the Zariski decomposition of $D$. If $\kappa(X,{-}K_X)\geq 1$, then $P$ is semiample.
\end{lem}

\begin{proof}
	By Lemma \ref{lem:Zariski}(e) we have $\kappa(X,P)=\kappa(X,D)=\kappa(X,-K_X)\geq 1$.
    
    If $\kappa(X,P)=2$, then $P$ is a nef and big divisor. If $N$ is the negative part in the Zariski decomposition of $D$, then $(X,N)$ is a klt pair by Lemma \ref{lem:prepared}. Thus, by \cite[Proposition 2.61]{KM98} there exists an effective $\Q$-divisor $E$ and an ample $\Q$-divisor $A$ such that $P\sim_\Q E+A$ and the pair $(X,N+2E+2A)$ is klt. Since $K_X+N+P\sim_\Q0$, we have $K_X+N+2E+2A\sim_\Q P$, hence $K_X+N+2E+2A$ is nef and big, thus semiample by the Abundance theorem for klt surface pairs.

    Now consider the case when $\kappa(X,P)=1$. Since $P$ is nef but not big, we have $\nu(X,P)\leq 1$, hence $\nu(X,P)=\kappa(X,P)=1$. By \cite[Proposition 2.1]{Kaw85} there exists a birational morphism $\mu\colon Y\to X$, a fibration $f\colon Y\to Z$, and a big and nef divisor $D$ on $Z$ such that
    \begin{equation}\label{eq:9000}
    \mu^*P\sim_\Q f^*D.
	\end{equation}
	In particular, we have
    $$ \dim Z=\kappa(Z,D)=\kappa(Y,f^*D)=\kappa(Y,\mu^*P)=\kappa(X,P)=1, $$
    hence $Z$ is a curve. Consequently, the divisor $D$ is ample, which implies that $P$ is semiample by \eqref{eq:9000}, as desired.
\end{proof}

Now we can finally prove Theorem \ref{thm:main0} and \ref{thm:main1}.

\begin{proof}[Proof of Theorem \ref{thm:main1}]
The first statement of the theorem follows immediately from Lemma \ref{lem:nonvanishing}(b). We prove parts (a), (b) and (c) in the following two steps.

\medskip

\emph{Step 1.}
In this step we prove (a).

Let $D\geq0$ be a $\Q$-Cartier divisor such that $D\sim_\Q{-}K_X$, and let $D=N+P$ be the Zariski decomposition of $D$. Then $(X,N+P)$ is a klt Calabi--Yau generalised pair by Lemma \ref{lem:prepared}, and $P\geq0$ by Lemma \ref{lem:Zariski}(d).

Assume first that $\kappa(X,{-}K_X)\geq1$ or that $P=0$. Then the divisor $P$ is semiample by Lemma \ref{lem:semiample}, and thus by Bertini's theorem there exists an effective divisor $P_1\sim_\Q P$ such that, if we set $B_1\coloneqq N+P_1$, then $(X,B_1)$ is a klt Calabi--Yau pair. We conclude that the first statement in (a) holds by \cite[Theorem 4.1]{Tot10}.

Now assume that $X$ is not rational. The previous paragraph deals with the case when $P=0$. Thus, we may assume that, additionally, $P\neq 0$. Then the cone $\Nef(X)$ is rational polyhedral by Corollary \ref{cor:nonrational}, and $\Nef^e(X)=\Nef(X)$ by Theorem \ref{thm:nefandeffectivecones}(b). This proves the second statement in (a) by \cite[Corollary 6.6]{Laz13}.

\medskip

\emph{Step 2.}
In this step we show (b) and (c).

Let $f\colon (X',B')\to(X,B)$ be the $\Q$-factorial terminalisation of $(X,B)$ as in Theorem \ref{thm:terminalisation}. Then $X'$ is smooth, and $(X',B'+f^*M)$ is a klt Calabi--Yau generalised pair. By Lemma \ref{lem:pushforward}(a) we have $\kappa(X',{-}K_{X'})\leq\kappa(X,{-}K_X)=0$, and by Lemma \ref{lem:nonvanishing}(b) we have $\kappa(X',{-}K_{X'})\geq0$. Therefore,
\begin{equation}\label{eq:kappa0}
\kappa(X',{-}K_{X'})=0.
\end{equation}
Let $D'\geq0$ be a $\Q$-Cartier divisor such that $D'\sim_\Q{-}K_{X'}$, and let $D'=N'+P'$ be the Zariski decomposition of $D'$. Then $(X',N'+P')$ is a klt Calabi--Yau generalised pair by Lemma \ref{lem:prepared}, and $P'\geq0$ by Lemma \ref{lem:Zariski}(d).

First assume that $P'=0$. Then $D'=N'$ and $(X',D')$ is a klt Calabi--Yau surface pair. Moreover, we have $ f_*D'\sim_\Q{-}f_*K_{X'}\sim_\Q{-}K_X\sim_\Q D $, and since $\kappa(X,D)=0$, we conclude that $D=f_*D'$. Furthermore, since $K_X+D\sim_\Q0$ and $K_{X'}+D'\sim_\Q0$, we have $K_{X'}+D'\sim_\Q f^*(K_X+D)$. Thus, by \cite[Lemma 2.30]{KM98} we conclude that the pair $(X,D)$ is also klt. Therefore, $\mathrm{(b)}_1$ holds by \cite[Theorem 4.1]{Tot10}.

Now assume that $P'\neq0$. Since $\kappa(X',P')=\kappa(X',{-}K_{X'})=0$ by Lemma \ref{lem:Zariski}(e) and by \eqref{eq:kappa0}, the divisor $P'$ is nef but not big, hence $(P')^2=0$. Then by Theorem \ref{thm:sak84}, there exists a birational morphism $\pi\colon X'\to X_0$ to a Sakai surface $X_0$. This shows $\mathrm{(b)}_2$.

Finally, assume that $\Delta_{X_0,D_0}^{\nod}=\emptyset$\footnote{Recall the definition of a \textit{generic} Sakai surface from Definition \ref{dfn:generic_Sakai}.} and that there exists a rational polyhedral fundamental domain for the action of $\Cr(X_0,D_0)$ on $\Nef^e(X_0)$. Then by Theorem \ref{thm:liftingoneblowup} there exists a rational polyhedral fundamental domain for the action of $\Cr(X',D')$ on $\Nef(X')$. Moreover, we have $\Delta_{X',D'}^{\nod}=\emptyset$ by Theorem \ref{thm:sak84}(f). Therefore, by Theorem \ref{thm:descent} there exists a rational polyhedral fundamental domain for the action of $\Cr(X,D)$ on $\Nef^e(X)$. This proves (c) and finishes the proof of the theorem.
\end{proof}

\begin{proof}[Proof of Theorem \ref{thm:main0}]
Let $(X,B+M)$ be a very general klt Calabi--Yau generalised pair of dimension $2$. This means that Theorem \ref{thm:main1}(a) holds, or that Theorem \ref{thm:main1}$\mathrm{(b)}_1$ holds, or that Theorem \ref{thm:main1}$\mathrm{(b)}_2$ holds, where $X_0$ is a generic Sakai surface. Therefore, by Theorem \ref{thm:main2}, which is proved in \cite{LSX26}, we conclude that there exists a rational polyhedral fundamental domain for the action of the group $G$ on $\Nef^e(X)$, where $G$ is either the representation of the group $\Aut(X)$ in $\GL\big(N^1(X)_\R\big)$, or it is the group $\Cr(X,D)$, where $D$ is as in Theorem \ref{thm:main1}(b). In both cases we have that $G$ is a subgroup of $\Cr(X)$.

Therefore, the action of $G$ on $\Nef^e(X)$ is of polyhedral type. Since $G\subseteq\Cr(X)$, this clearly implies that the action of $\Cr(X)$ on $\Nef^e(X)$ is of polyhedral type. But then by Theorem \ref{thm:Looijenga} we conclude that there exists a rational polyhedral fundamental domain for the action of the group $\Cr(X)$ on $\Nef^e(X)$, as desired.
\end{proof}

\newpage

\appendix

\section{On Cremona isometries}

In this appendix, we show that for a generalised Halphen surface $X$, several conditions in the definition of a Cremona isometry of $X$ from Definition \ref{dfn:cremona} are in fact redundant.

\begin{prop}\label{prop:groups}
Let $X$ be a generalised Halphen surface such that the cone $\Eff(X)$ is not rational polyhedral. Let $\sigma$ be an automorphism of the vector space $N^1(X)_\R$ such that:
    \begin{enumerate}[\normalfont (i)]
        \item $\sigma$ preserves the intersection form on $N^1(X)_\R$,
        \item $\sigma$ preserves the effective cone $\Eff(X)$.
    \end{enumerate}
    Then $\sigma\in\Cr(X)$.
\end{prop}

\begin{proof}
Let $\mathcal N(X)\subseteq N^1(X)$ be the set of numerical classes of prime divisors on $X$ with negative self-intersection. Let $D \in |{-}K_X|$ be a divisor of canonical type. 

\medskip

\emph{Step 1.}
In this step we show that there exists $p>0$ such that $\sigma\big([K_X]\big)=p[K_X]$.

By Lemma \ref{lem:char_of_indecomp}(b), there exists a birational morphism $\pi\colon X\to\mathbb P^2$. Then $\pi_*K_X\sim K_{\mathbb P^2}$, hence $K_X$ is not torsion. By Lemma \ref{lem:char_of_indecomp}(a), the divisor ${-}K_X$ is nef. Therefore,  since the cone $\Eff(X)$ is not rational polyhedral by assumption, by \cite[Proposition 6.2]{Bor91} we conclude that there is a sequence $\{R_n\}_{n\in\N}$ of extremal rays of $\Eff(X)$ which accumulate to $\R_{\geq0}[{-}K_X]$. Then the sequence of rays $\{\sigma(R_n)\}_{n\in\N}$ accumulates to the ray $\R_{\geq0}\sigma\big([{-}K_X]\big)$. On the other hand, since each $\sigma(R_n)$ is an extremal ray of $\Eff(X)$ by assumption (ii), we conclude by op.\ cit.\ that the sequence of rays $\{\sigma(R_n)\}_{n\in\N}$ accumulates to the ray $\R_{\geq0}[{-}K_X]$. Thus, $\R_{\geq0}\sigma\big([{-}K_X]\big)=\R_{\geq0}[{-}K_X]$, hence the claim.

\medskip

\emph{Step 2.}
We claim that for each $\alpha\in\mathcal N(X)$ we have $\alpha^2\in\{-1,-2\}$. Moreover, if $\alpha^2={-}2$, then $\alpha\cdot [K_X]=0$, and if $\alpha^2={-}1$, then $\alpha\cdot [K_X]={-}1$. 

Indeed, let $G$ be a prime divisor on $X$ such that $G^2<0$. Since ${-}K_X$ is nef by Lemma \ref{lem:char_of_indecomp}(a), we conclude that $G^2\in\{-1,-2\}$ by Lemma \ref{lem:negativeselfintersection}(a) applied for $B\coloneqq 0$. If $G^2={-}1$, then $G$ is a $({-}1)$-curve by Lemma \ref{lem:negativeselfintersection}(c). If $G^2={-}2$, then $G\cdot K_X=0$ by Lemma \ref{lem:negativeselfintersection}(b) applied for $B\coloneqq 0$.

\medskip

\emph{Step 3.}
Let $\alpha\in\mathcal N(X)$ with $\alpha^2={-}1$. In this step we show that $\sigma(\alpha)\in\mathcal N(X)$.

Since $\sigma(\alpha)\in\Eff(X)$ by assumption (ii), there exists an effective $\R$-divisor $G$ on $X$ such that
$$\sigma(\alpha)=[G].$$
Since $\alpha^2<0$, by Lemma \ref{lem:negativeexceptional} we have that the ray $\R_{\geq0}\alpha$ is an extremal ray of the cone $\Eff(X)$, and therefore, the ray $\R_{\geq0}[G]=\R_{\geq0}\sigma(\alpha)$ is an extremal ray of the cone $\Eff(X)$. This forces the support of $G$ to be irreducible, so we can write 
$$G=f F, \quad\text{where $f>0$ and $F$ is a prime divisor on $X$}.$$
Then by assumption (i) and by Step 2 we have 
$$ f^2 F^2=G^2=\sigma(\alpha)^2=\alpha^2={-}1\quad\text{and}\quad F^2\in\{-1,-2\}. $$
If $F^2={-}1$, then $f=1$ by above, and hence $\sigma(\alpha)=[F]\in\mathcal N(X)$.

Thus, we may assume that $F^2={-}2$, hence $K_X\cdot F=0$ and $[K_X]\cdot \alpha={-}1$ by Step 2. As $\sigma\big([K_X]\big)=p[K_X]$ for some $p>0$ by Step 1, we conclude by assumption (i) that
$$0=pf[K_X]\cdot[F]=\sigma\big([K_X]\big)\cdot \sigma(\alpha)=[K_X]\cdot \alpha={-}1,$$
a contradiction.

\medskip

\emph{Step 4.}
Let $\alpha\in\mathcal N(X)$ with $\alpha^2={-}2$. In this step we show that $\sigma(\alpha)\in\mathcal N(X)$.

As in Step 3 we can show that
$$\sigma(\alpha)=f[F], \quad\text{where $f>0$ and $F$ is a prime divisor on $X$}.$$
Then by assumption (i) and by Step 2 we have 
\begin{equation}\label{eq:equalities2}
f^2 F^2=\sigma(\alpha)^2=\alpha^2={-}2\quad\text{and}\quad F^2\in\{-1,-2\}.
\end{equation}
If $F^2={-}2$, then $f=1$ by above, and hence $\sigma(\alpha)=[F]\in\mathcal N(X)$.

Thus, we may assume that $F^2={-}1$. By Step 3 we have $\sigma^{-1}\big([F]\big)\in\mathcal N(X)$, hence there exists a prime divisor $P$ on $X$ such that $\sigma^{-1}\big([F]\big)=[P]$. Since $\alpha\in\mathcal N(X)$, there exists a prime divisor $Q$ on $X$ such that $\alpha=[Q]$. Therefore,
$$[Q]=\alpha=\sigma^{-1}\big(f[F]\big)=[fP],$$
hence $Q=fP$ by Lemma \ref{lem:negativeexceptional}. As $P$ and $Q$ are prime divisors, this forces $f=1$, which contradicts \eqref{eq:equalities2} since $F^2={-}1$.

\medskip

\emph{Step 5.}
In this step we show that $\sigma\big([K_X]\big)=[K_X]$.

By Step 1 there exists $p>0$ such that $\sigma\big([K_X]\big)=p[K_X]$. Since $X$ is the blowup of $\mathbb{P}^2$ at nine points by Lemma \ref{lem:char_of_indecomp}(b), there exists $\beta\in N^1(X)$ which is the numerical class of a $({-}1)$-curve on $X$. By Step 3 we have $\sigma(\beta)\in\mathcal N(X)$, and since $\sigma(\beta)^2=\beta^2={-}1$ by assumption (i), we have $[K_X]\cdot\sigma(\beta)={-}1$ by Step 2. Hence, assumption (i) and Step 2 give
$${-}p=p[K_X]\cdot\sigma(\beta)=\sigma\big([K_X]\big)\cdot \sigma(\beta)=[K_X]\cdot \beta={-}1.$$
This proves the claim.

\medskip

\emph{Step 6.}
Finally, in this step we show that $\sigma$ maps the lattice $N^1(X)$ onto itself.

Consider first $\alpha\in \Eff(X)\cap N^1(X)$. Then by \cite[Lemma 4.1]{LH05},\footnote{Note that \cite[Lemma 4.1]{LH05} holds under the assumption that $X$ is an \emph{anticanonical} rational surface, which means that $|{-}K_X|\neq\emptyset$.}  there exist classes $\alpha_1,\dots,\alpha_n\in\mathcal N(X)$ and positive integers $a_0,a_1,\dots,a_n$ such that $\alpha=a_0[{-}K_X]+a_1\alpha_1+\dots +a_n\alpha_n$. But then $\sigma(\alpha)=a_0\sigma\big([{-}K_X]\big)+ a_1\sigma(\alpha_1)+\dots +a_n\sigma(\alpha_n)$, which belongs to $N^1(X)$ by Steps 3, 4 and 5.

Now let $\alpha\in N^1(X)$. Then there exist ample integral classes $\beta$ and $\gamma$ such that $\alpha=\beta-\gamma$, hence $\sigma(\alpha)=\sigma(\beta)-\sigma(\gamma)\in N^1(X)$ by the previous paragraph. This concludes the proof.
\end{proof}

\section{Generic Sakai surfaces}\label{sec:generic_good_Halphen}

Let $X$ be a Sakai surface of type $R$. The geometry of $X$ depends not only on its type $R$, but also on the configuration of blownup points. In particular, the existence of internal $(-2)$-curves depends on this configuration. We introduce the following definition.

\begin{dfn}\label{dfn:generic_Sakai}
Let $X$ be a Sakai surface and let $D\in|{-}K_X|$. We say that $X$ is \emph{generic} if $\Delta^{\nod}_{X,D} = \emptyset$. 
\end{dfn}

The choice of the term \emph{generic} is justified by the following Proposition \ref{prop:noNodalOnGen1}. As explained in detail in \cite[Section 5 and Appendix B]{Sak01}, for a fixed type $R$, the surface $X$ depends on several parameters which are complex numbers, and by varying these parameters we obtain all possible $R$-surfaces as fibres of a family over an affine base. Then the main result of this appendix, Proposition \ref{prop:noNodalOnGen1}, says that a generic Sakai surface $X$ of type $R$ is a very general fibre of the deformation from the previous sentence.

Before proving Proposition \ref{prop:noNodalOnGen1}, we first need to show that for every type $R$ there exists a Sakai surface of type $R$. We start with the following lemma.

\begin{lem}\label{lem:mult_types}
    Let $X$ be an $R$-surface, where the type $R$ of $X$ is either elliptic or multiplicative, and let $D\in |{-}K_X|$. By Lemma \ref{lem:char_of_indecomp}(b), there exists a morphism $\mu\colon X\to \pr^2$ which is the composition of blowups of nine (possibly infinitely close) points $p_1,\ldots, p_9$, as in the blowup configurations from \cite[Appendix B]{Sak01}. Then the following holds.
    \begin{enumerate}[\normalfont (a)]
        \item If $D$ has at most three components, i.e.\ if $R\in\{A_0^\smone, A_0^\smonestar, A_1^\smone, A_2^\smone \}$, then $\mu$ induces an isomorphism $D\simeq \mu_* D$ and $p_1,\ldots,p_9$ lie on the smooth locus of $\mu_* D$.
        Up to projective transformations of $\mathbb P^2$, the divisor $\mu_*D$ depends only on the type $R$.
        \item Otherwise, we have $R\in \{ A_3^\smone, A_4^\smone, A_5^\smone, A_6^\smone, A_7^\smone,{A_7^\smone}', A_8^\smone\}$. Then there exist a permutation $\sigma\in S_9$ and an integer $2\leq k\leq7$ such that the following conditions are satisfied:
        \begin{enumerate}[\normalfont (i)]
            \item $\mu$ factors as $X\xrightarrow{\lambda} Y\xrightarrow{\nu}\pr^2$, where $\nu$ is the composition of the blowups of $p_{\sigma(1)},\ldots,p_{\sigma(k-1)}$ and $\lambda$ is the composition of the blowups of $p_{\sigma(k)},\ldots, p_{\sigma(9)}$,
            \item $\lambda$ induces an isomorphism $D\simeq \lambda_* D$,
            \item $p_{\sigma(k)},\ldots, p_{\sigma(9)}$ lie on the smooth locus of $\lambda_* D$.
        \end{enumerate}
        Up to projective transformations of $\mathbb P^2$, the surface $Y$ and the divisor $\lambda_*D$ depend only on the type $R$.
    \end{enumerate} 
\end{lem}
\begin{proof}
    Part (a) follows directly from \cite[Appendix B]{Sak01}.
    
    For part (b), we explain in detail the case $R=A_3^\smone$ as in \cite[Appendix B]{Sak01}. In this case, up to a change of coordinates, $X$ is obtained by blowing up points on the cycle $C$ of three lines $\{x=0\}$, $\{y=0\}$ and $\{z=0\}$ in $\mathbb P^2$; the points $p_1,\ldots,p_7$ lie on the smooth locus of $C$, and the point $p_9$ lies on the exceptional curve over $p_8$, but not on the strict transforms of the lines $\{x=0\}$ and $\{z=0\}$. In coordinates, the points are given by
\begin{align*}
    &p_1\colon (a_1:0:1),\quad p_2\colon (a_2:0:1),\quad p_3\colon (a_3:0:1), \quad p_4\colon (0:1:a_4),\\
    &p_5\colon (0:1:a_5),\quad p_6\colon (1:a_6:0),\quad p_7\colon (1:a_7:0), \quad p_8 \colon (0:1:0), \\
    &p_9\in \textstyle \Spec \C\big[\frac{x}{y},\frac{z}{y}\big/\frac{x}{y}\big], \quad p_9\colon \big(\frac{x}{y},\frac{z}{x}\big) = (0,a_8),
\end{align*}
where $a_i\neq 0$ for all $i\in\{1,\ldots,8 \}$.

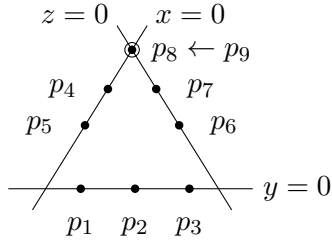
\begin{center}    
\begin{tikzpicture}[scale=0.8]

\draw (-0.15,0) -- (3.9,0) node[right] {$y=0$};
\draw (0.25,-0.35) -- (2.15,2.7);
\draw (3.55,-0.35) -- (1.65,2.7);

\fill (1.05,0) circle (2pt) node[below=7pt] {$p_1$};
\fill (1.95,0) circle (2pt) node[below=7pt] {$p_2$};
\fill (2.85,0) circle (2pt) node[below=7pt] {$p_3$};

\fill (1.12,1.05) circle (2pt) node[left=8pt] {$p_5$};
\fill (1.50,1.65) circle (2pt) node[left=8pt] {$p_4$};

\fill (2.68,1.05) circle (2pt) node[right=8pt] {$p_6$};
\fill (2.30,1.65) circle (2pt) node[right=8pt] {$p_7$};

\draw (1.9,2.30) circle (0.12);
\fill (1.9,2.30) circle (2pt);

\node[above left=2pt] at (1.77,2.48) {$z=0$};
\node[above right=2pt] at (2.03,2.48) {$x=0$};

\node at (3.05,2.25) {$p_8 \leftarrow p_9$};
\end{tikzpicture}
\end{center}

Let $\sigma\in S_9$ be the cycle $\begin{pmatrix}1 & 8 & 7 & 6 & 5 & 4 & 3 & 2 \end{pmatrix}$. If we set $k\coloneqq 2$ and define the maps $\lambda$ and $\nu$ as in (i), we obtain the factorisation $\mu=\nu\circ\lambda$, and (i) and (iii) hold.

Thus, $Y$ is the blowup of $\mathbb P^2$ at $p_{\sigma(1)}=p_8$. Let $E_1$ be the $\nu$-exceptional divisor and set $C_Y\coloneqq \nu^*C - E_1$. Then $C_Y\in |{-}K_Y|$ is a cycle of four rational curves. For $i\in\{2,\ldots,9\}$, let $E_i$ be the exceptional divisor on $X$ over $p_{\sigma(i)}$. Then ${-}K_X\sim {-}\lambda^*K_Y-\sum_{i=2}^9 E_i\sim \lambda^*C_Y-\sum_{i=2}^9 E_i$. Then the unique divisor $D\in |{-}K_X|$ is precisely the strict transform of $C_Y$, hence (ii) is satisfied.

The remaining cases in (b) are analogous. Here we only note the numbers $k$, the permutations $\sigma\in S_9$ and the formula for the divisor $\lambda_*D$ in each case.

If $R=A_4^\smone$, then $X$ is obtained by blowing up points on the cycle $C$ of three lines $\{x=0\}$, $\{y=0\}$ and $\{z=0\}$ in $\mathbb P^2$. Then $k=3$, $\sigma=\begin{pmatrix}1 & 7 & 5 & 3\end{pmatrix}\begin{pmatrix}2 & 8 & 6 & 4\end{pmatrix}$ and $\lambda_*D=\nu^*C - E_1-E_2$, where $E_1$ and $E_2$ are the $\nu$-exceptional divisors.

If $R=A_5^\smone$, then $X$ is obtained by blowing up points on the cycle $C$ of three lines $\{x=0\}$, $\{y=0\}$ and $\{z=0\}$ in $\mathbb P^2$. Then $k=4$, $\sigma=\begin{pmatrix}1 & 3 & 8 & 6 & 4\end{pmatrix}\begin{pmatrix}2 & 7 & 5 \end{pmatrix}$ and $\lambda_*D=\nu^*C - E_1-E_2-E_3$, where $E_1$, $E_2$ and $E_3$ are the $\nu$-exceptional divisors.

If $R=A_6^\smone$, then $X$ is obtained by blowing up points on the cycle $C$ of three lines $\{x=0\}$, $\{y=0\}$ and $\{z=0\}$ in $\mathbb P^2$. Then $k=5$, $\sigma=\begin{pmatrix}1 & 2 & 3 & 7 & 5\end{pmatrix}\begin{pmatrix}4 & 8 & 6 \end{pmatrix}$ and $\lambda_*D=\nu^*C - E_1-E_2-E_3-E_4$, where $E_1$, $E_2$, $E_3$ and $E_4$ are the $\nu$-exceptional divisors.

If $R=A_7^\smone$, then $X$ is obtained by blowing up points on the cycle $C$ of three lines $\{x=0\}$, $\{y=0\}$ and $\{z=0\}$ in $\mathbb P^2$. Then $k=6$, $\sigma=\begin{pmatrix}1 & 2 & 3 & 4 & 7 & 5 & 8 & 6\end{pmatrix}$ and $\lambda_*D=\nu^*C - E_1-E_2-E_3-E_4-E_5$, where $E_1$, $E_2$, $E_3$, $E_4$ and $E_5$ are the $\nu$-exceptional divisors.

If $R={A_7^\smone}'$, then $X$ is obtained by blowing up points on the cycle $C$ of three lines $\{x=0\}$, $\{y=0\}$ and $\{z=0\}$ in $\mathbb P^2$. Then $k=6$, $\sigma=\begin{pmatrix}2 & 3 \end{pmatrix}\begin{pmatrix}4 & 7 & 5 & 8 & 6\end{pmatrix}$ and $\lambda_*D=\nu^*C - E_1-E_2-E_3-E_4-E_5$, where $E_1$, $E_2$, $E_3$, $E_4$ and $E_5$ are the $\nu$-exceptional divisors.

If $R=A_8^\smone$, then $X$ is obtained by blowing up points on the cycle $C$ of three lines $\{x=0\}$, $\{y=0\}$ and $\{z=0\}$ in $\mathbb P^2$. Then $k=7$, $\sigma=\begin{pmatrix}5 & 7 \end{pmatrix}\begin{pmatrix}6 & 8\end{pmatrix}$ and $\lambda_*D=\nu^*C - E_1-E_2-E_3-E_4-E_5-E_6$, where $E_1$, $E_2$, $E_3$, $E_4$, $E_5$ and $E_6$ are the $\nu$-exceptional divisors.
\end{proof}

Using the concrete blowup description from Lemma \ref{lem:mult_types}, we can show the following result which is stronger than the existence of a Sakai surface of elliptic or multiplicative type, and which is also of independent interest.

\begin{lem}\label{lem:generic_blowup}
Let $X$ be an $R$-surface, where the type $R$ of $X$ is 
  either elliptic or multiplicative.
   With notation from Lemma \ref{lem:mult_types}, set $Y\coloneqq \pr^2$, $k=1$ and $\lambda\coloneqq \mu$ in the case of Lemma \ref{lem:mult_types}(a). Let $q_k,\ldots, q_9$ be very general points in the smooth locus of $C_Y = \lambda_* D$ and let $\lambda'\colon X'\to Y$ be the composition of the blowups of $q_k,\ldots, q_9$, then $X'$ is a Sakai surface of type $R$.
\end{lem}

\begin{proof}
For each $i\in\{k,\ldots,9\}$, pick a point $q_i$ lying in the smooth locus $(C_Y)_{\sm}$ of $C_Y$ such that for each component $C_\ell$ of $C_Y$, there are exactly $C_Y\cdot C_\ell \geq 0$ points lying on $C_\ell$.

Let $\lambda'\colon X'\to Y$ be the composition of the blowups of $q_k,\ldots, q_9$. Then $X'$ is a generalised Halphen surface. 
Let $D'$ be the strict transform of $C_Y$ on $X'$ and let $E'_i$ be the exceptional divisor on $X'$ over $q_i$ for each $i\in\{k,\ldots,9\}$. Then 
\begin{equation}\label{eq:D'_lin_equiv}
    \textstyle D'\sim \lambda'^*C_Y-\sum_{i=k}^9 E'_i\sim {-}\lambda'^*K_Y-\sum_{i=k}^9 E'_i\sim{-}K_{X'}.
\end{equation}
Now, if $q_k,\ldots, q_8$ are chosen as in the first paragraph, there exists a unique component $C_{\ell_9}$ of $C_Y$ such that $C_Y\cdot C_{\ell_9}>0$. We will show that for a very general point $q_9\in (C_{\ell_9})_{\sm}$, the surface $X'$ obtained by blowing up $q_k,\ldots, q_9$ is a Sakai surface, i.e.\ we have $h^0\big(X',\OO_{X'}(nD')\big) = 1$ for any $n\geq 1$.

\medskip

\emph{Step 1.}
In this step we show that it suffices to prove that 
$\OO_{D'}(D')\in \Pic^0(D')$ is not torsion for a very general point $q_9\in (C_{\ell_9})_{\sm}$.

We follow the proof of \cite[Proposition 2.2]{CD12}. Recall that $D'$ is a curve of canonical type. Hence, for any $n\geq 1$, the restriction of $\OO_{D'}(nD')$ to each irreducible component of $D'$ is an invertible sheaf of degree zero and thus $\OO_{D'}(D')\in \Pic^0(D')$. Since $D'$ is a curve of arithmetic genus one, we can apply the Riemann-Roch theorem to $D'$ and we conclude that
\begin{equation}\label{eq:equivalences_D'}
\begin{aligned}
    h^0\big(D',\OO_{D'}(n D')\big)>0 &\iff \OO_{D'}(nD')\simeq \OO_{D'} \\
    &\iff h^0\big(D',\OO_{D'}(n D')\big)=1.
\end{aligned}
\end{equation}
Since $X'$ is rational, we have $h^1(X',\OO_{X'})=0$ and $\chi(\OO_{X'})=1$. Since $D'\sim -K_X$ and $D'^2=0$, the Riemann-Roch formula and the Serre duality give 
\begin{equation}\label{eq:h1_nD'}
    h^1\big(X',\OO_{X'}(nD')\big) = h^0\big(X',\OO_{X'}(nD')\big) -1 \text{ for any }n\geq 1.
\end{equation}
Now the short exact sequence
\begin{equation}\label{eq:ses_D'}
0 \to \OO_{X'}\big( (n-1) D' \big) \to \OO_{X'}(nD') \to \OO_{D'}(nD')\to 0
\end{equation}
yields in particular $h^0\big(X',\OO_{X'}(D')\big)=1$.

If $\OO_{D'}(D')\in \Pic^0(D')$ is not torsion, then by \eqref{eq:equivalences_D'} we have 
\[h^0\big(D',\OO_{D'}(n D')\big)=0 \text{ for any }n\geq 1.\] Thus \eqref{eq:ses_D'} implies $h^0\big(X',\OO_{X'}(nD')\big) = h^0(X',\OO_{X'})=1.$

If $\OO_{D'}(D')\in \Pic^0(D')$ is torsion, let $m$ be the smallest positive integer such that $\OO_{D'}(mD')\simeq \OO_{D'}$.
By \eqref{eq:ses_D'} and \eqref{eq:equivalences_D'}, we obtain
\[h^0\big(X',\OO_{X'}(nD')\big)=1 \text{ for any }n\in \{0,\dots,m-1\}.\] In particular, it follows from \eqref{eq:h1_nD'} that $h^1\big(X',\OO_{X'}\big((m-1)D'\big)\big)=0$. Therefore, by \eqref{eq:ses_D'} and the above equality, and since $\OO_{D'}(mD')\simeq \OO_{D'}$, we obtain 
\[h^0\big(X',\OO_{X'}(mD')\big) = h^0\big(X',\OO_{X'}\big((m-1)D'\big)\big) + h^0\big(D',\OO_{D'}(mD')\big) =2.\]
Thus $X'$ is not a Sakai surface, which concludes the proof.

\medskip

\emph{Step 2.}
Note that by \eqref{eq:D'_lin_equiv} we have
\begin{align*}
\OO_{D'}(D')&\simeq \textstyle \OO_{D'}(\lambda'^*C_Y)\otimes\OO_{D'}\big({-}\sum_{i=k}^9 E'_i\big)\\
&\simeq \textstyle (\lambda'|_{D'})^*\OO_{C_Y}(C_Y)\otimes\OO_{D'}\big({-}\sum_{i=k}^9 E'_i\big).
\end{align*} 
Since the map $\lambda'|_{D'}\colon D'\to C_Y$ is an isomorphism and all the points $q_i$ lie in $(C_Y)_{\sm}$, we obtain 
$\OO_{D'}\big({-}\sum_{i=k}^9 E'_i\big)=(\lambda'|_{D'})^*\OO_{C_Y}\big({-}\sum_{i=k}^9 q_i\big)$,
which together with the relation above gives
$$\OO_{D'}(D')\simeq \textstyle (\lambda'|_{D'})^*\big(\OO_{C_Y}(C_Y)\otimes \OO_{C_Y}({-}\sum_{i=k}^9 q_i)\big).$$
Therefore, by Step 1 and since $\lambda'|_{D'}$ is an isomorphism, it remains to show that $\OO_{C_Y}(C_Y)\otimes \OO_{C_Y}({-}\sum_{i=k}^9 q_i)$ is not torsion in $\Pic^0(C_Y)$ for a very general point $q_9\in (C_{\ell_9})_{\sm}$.

\medskip

\emph{Step 3.} In this step we show that $\OO_{C_Y}(C_Y)\otimes \OO_{C_Y}({-}\sum_{i=k}^9 q_i)\in \Pic^0(C_Y)$ is not torsion for a very general point $q_9\in (C_{\ell_9})_{\sm}$.

If $R=A_0^\smone$, then $Y=\pr^2$ and $C_Y$ is a smooth elliptic curve. Since the Abel-Jacobi map induces an isomorphism $ \Pic^0(C_Y)\simeq C_Y$, the line bundle $\OO_{C_Y}(C_Y)\otimes \OO_{C_Y}({-}\sum_{i=k}^9 q_i)\in \Pic^0(C_Y)$ is not torsion for a very general point  $q_9\in C_Y$.
If $R\neq A_0^\smone$, then by \cite[Lemma 1.7(iv)]{Fri15}, there exists a unique point $q\in (C_{\ell_9})_{\sm}$ such that 
$$
\textstyle \OO_{C_Y}(C_Y)\otimes \OO_{C_Y}({-}\sum_{i=k}^8 q_i) \simeq \OO_{C_Y}(q).
$$
By \cite[Definition 1.12 and Lemma 1.13]{Fri13}, the Abel-Jacobi map induces an isomorphism $\alpha\colon \Pic^0(C_Y) \to (C_{\ell_9})_{\sm}\simeq \C^*$; in particular, we have 
$$
\textstyle\alpha \big( \OO_{C_Y}(C_Y)\otimes \OO_{C_Y}({-}\sum_{i=k}^9 q_i)\big) = \alpha \big( \OO_{C_Y}(q - q_9) \big) = q_9/q.
$$
Therefore, for a very general point $q_9 \in (C_{\ell_9})_{\sm}$ such that $q_9/q\in \C^*$ is not a root of unity, the line bundle $\OO_{C_Y}(C_Y)\otimes \OO_{C_Y}({-}\sum_{i=k}^9 q_i)\in \Pic^0(C_Y)$ is not torsion, as desired.
\end{proof}

As an immediate consequence, we can show the existence of Sakai surface of type $R$, for every type $R$.
\begin{lem}\label{lem:atleastoneSakai}
    Let $R$ be one of the types in Table \ref{table:typeR}.
	\begin{enumerate}[\normalfont (a)]
	\item If $R$ is an elliptic type or a multiplicative type, then
     there exists a Sakai surface of type $R$.
    \item If $R$ is an additive type, then every $R$-surface is a Sakai surface.
	\end{enumerate}
\end{lem}

\begin{proof}
Part (a) follows from Lemma \ref{lem:generic_blowup}, while part (b) follows from Lemma \ref{lem:char_of_indecomp}(d). 
\end{proof}
 
We are now ready to prove the following main result of this appendix.

	\begin{prop}\label{prop:noNodalOnGen1}
	Fix a type $R$ of a Sakai surface. There exist an analytic variety $S$, an analytic family $\pi\colon \mathcal X\to S$ and subsets $S_\mathrm{Sak}$ and $S_\mathrm{genSak}$ of $S$ whose complements in $S$ are countable unions of closed analytic subvarieties, such that the following holds.
\begin{enumerate}[\normalfont (a)]
\item Every fibre of $\pi$ is an $R$-surface.
\item Every fibre of $\pi$ over $S_\mathrm{Sak}$ is a Sakai surface of type $R$.
\item Every fibre of $\pi$ over $S_\mathrm{genSak}$ is a generic Sakai surface of type $R$.
\end{enumerate}
	\end{prop}
		
\begin{proof}
Throughout the proof we use without explicit mention that for a Sakai surface of type $R$ such that $D\in|{-}K_X|$, we have $\Delta^{\nod}_{X,D}\subseteq Q(R^\smperp)$, which is immediate from the definition of $\Delta^{\nod}_{X,D}$. In particular, every element of $\Delta^{\nod}_{X,D}$ can be written as an integral linear combination of the elements in $\Pi(R^\smperp)$. We denote by $D_{\operatorname{red}}$ the support of $D$.

\medskip

\emph{Step 1.}
We will use the period map
$$\chi\colon Q(R^\smperp)\to \C\mod \widehat \chi\big(H^1(D_{\operatorname{red}},\Z)\big)$$
constructed in \cite[pp.\ 187--188]{Sak01}, and normalised as in \cite[p.\ 188]{Sak01}. We suppress $\widehat \chi\big(H^1(D_{\operatorname{red}},\Z)\big)$ from the notation below, but we use implicitly that it is a lattice in $\C$, and in particular countable as a set.

If $C$ is an internal $(-2)$-curve on a Sakai surface $X$ of type $R$, we consider
$$\gamma\coloneqq [C]\in \Delta^{\nod}_{X,D}.$$
Since $C\subseteq X\setminus D_{\operatorname{red}}$ and the $2$-form $\omega$ used to define $\widehat \chi$ is holomorphic on $X\setminus D_{\operatorname{red}}$, we have
\begin{equation}\label{eq:chigamma1}
\chi(\gamma)=0.
\end{equation}

\medskip

\emph{Step 2.}
In this step we prove the proposition when $R=A_0^\smone$; when $R=A_0^\smonestar$, the proof is analogous.

We have $\Pi(R^\smperp)=\{\alpha_0,\dots,\alpha_8\}$. By \cite[Appendix B]{Sak01} we have
$$\chi(\alpha_0) = \theta_9 - \theta_8,\ \chi(\alpha_i) = \theta_{i+1}-\theta_i \text{ for }i = 1, \dots , 7,\  \chi(\alpha_8) = \theta_1 + \theta_2 + \theta_3,$$
where $\theta_i$ are complex parameters. Set $S\coloneqq \Spec\C[\theta_0,\dots,\theta_8]$, where we now view the parameters $\theta_i$ as variables over $\C$. Then the existence of the map $\pi\colon\mathcal X\to S$ such that (a) holds follows from the constructions in \cite[Section 5 and Appendix B]{Sak01}. Let $S^\circ\subseteq S$ be a non-empty open set such that the map $\pi|_{\pi^{-1}(S^\circ)}$ is flat.

For each positive integer $m$, let $S_m\subseteq S^\circ$ be the locus of points $s$ such that $h^0(\mathcal X_s,{-}mK_{\mathcal X_s})=1$, where $\mathcal X_s\coloneqq \pi^{-1}(s)$. Then $S_m\neq\emptyset$ by Lemma \ref{lem:atleastoneSakai}, and $S_m$ is open by Grauert's upper semicontinuity of cohomology. Setting $S_\mathrm{Sak}\coloneqq \bigcap_{m\geq1}S_m$, we have that (b) holds by Lemma \ref{lem:char_of_indecomp}(c).

For $\gamma$ as in Step 1, write $\gamma=m_0\alpha_0+\dots+m_8\alpha_8$, where $m_i\in\Z$. Then \eqref{eq:chigamma1} and the formulas above give
$$ \textstyle 0=\chi(\gamma)=\sum_{i=1}^7m_i(\theta_{i+1}-\theta_i)+m_8(\theta_1+\theta_2+\theta_3)+m_0(\theta_9-\theta_8). $$
Then for each $(m_0,\dots,m_8)\in\Z^9$, the set of solutions $(\theta_1,\dots,\theta_9)\in\C^9$ of the equation above is the union of countably many linear subspaces.\footnote{Recall that in the notation we suppress the lattice $\widehat \chi\big(H^1(D_{\operatorname{red}},\Z)\big)$.} Therefore, if $S_\mathrm{genSak}$ is the complement in $S_\mathrm{Sak}$ of that set of solutions, the corresponding surfaces $\pi^{-1}(s)$ for $s\in S_\mathrm{genSak}$ are Sakai surfaces without internal $(-2)$-curves. This shows (c).

\medskip

\emph{Step 3.}
In this step we prove the proposition when $R=A_1^\smone$; when $R\in\{A_2^\smone, A_3^\smone, A_4^\smone, A_7^\smone, {A_7^\smone}'\} $, the proof is analogous.

We have $\Pi(R^\smperp)=\{\alpha_0,\dots,\alpha_7\}$. By \cite[Appendix B]{Sak01} we have
\begin{align*}
e^{2\pi\sqrt{-1}\chi(\alpha_0)} &= a_8/a_9,  &e^{2\pi\sqrt{-1}\chi(\alpha_1)} = a_3/a_2, && e^{2\pi\sqrt{-1}\chi(\alpha_2)} &= a_2/a_1, \\
e^{2\pi\sqrt{-1}\chi(\alpha_3)} &= {-}a_1/a_4a_5,  &e^{2\pi\sqrt{-1}\chi(\alpha_4)} = a_5/a_6, && e^{2\pi\sqrt{-1}\chi(\alpha_5)} &= a_6/a_7,\\
e^{2\pi\sqrt{-1}\chi(\alpha_6)} &= a_7/a_8,  &e^{2\pi\sqrt{-1}\chi(\alpha_7)} = a_4/a_5, && &
\end{align*}
where $a_i$ are complex parameters. Set $S\coloneqq \Spec\C[a_1^{\pm1}, \dots , a_9^{\pm1}]$, where we now view the parameters $a_i$ as variables over $\C$. Then (a) and (b) follow as in Step 2.

For $\gamma$ as in Step 1, write $\gamma=m_0\alpha_0+\dots+m_7\alpha_7$, where $m_i\in\Z$. Then \eqref{eq:chigamma1} and the formulas above give
\begin{align*}
1=e^{2\pi\sqrt{-1}\chi(\gamma)}&=\frac{(-1)^{m_3}}{a_9^{m_0}}a_1^{m_3-m_2}a_2^{m_2-m_1}a_3^{m_1}a_4^{m_7-m_3}\\
&\times a_5^{m_4-m_3-m_7}a_6^{m_5-m_4}a_7^{m_6-m_5}a_8^{m_0-m_6}.
\end{align*}
Then it is clear that the set of solutions $(a_1,\dots,a_9)\in S$ of the equation above is the union of countably many analytic subsets. We conclude as in Step 2.

\medskip

\emph{Step 4.}
In this step we prove the proposition when $R=A_0^\smonestarstar$; when $R\in\{A_1^\smonestar, A_2^\smonestar\}$, the proof is analogous.

We have $\Pi(R^\smperp)=\{\alpha_0,\dots,\alpha_8\}$. By \cite[Appendix B]{Sak01} we have
$$\textstyle \chi(\alpha_0) = \frac{a_9 - a_8}{\lambda},\quad \chi(\alpha_i) = \frac{a_{i+1}-a_i}{\lambda} \text{ for }i = 1, \dots , 7,\quad  \chi(\alpha_8) = \frac{a_1 + a_2 + a_3}{\lambda},$$
where $a_i$ are complex parameters and we set $\lambda \coloneqq a_1+\dots+a_9$. Set $ S\coloneqq \Spec\C[a_1,\dots,a_9,\lambda^{-1}]$, where we now view the parameters $a_i$ as variables over $\C$. Then (a) and (b) follow as in Step 2.

For $\gamma$ as in Step 1, write $\gamma=m_0\alpha_0+\dots+m_8\alpha_8$, where $m_i\in\Z$. Then \eqref{eq:chigamma1} and the formulas above give
$$ \textstyle 0=\chi(\gamma)=\frac{1}{\lambda}\big(\sum_{i=1}^7m_i(a_{i+1}-a_i)+m_8(a_1+a_2+a_3)+m_0(a_9-a_8)\big).$$
Then it is clear that the set of solutions $(a_1,\dots,a_9)\in S$ of the equation above is the union of countably many linear subspaces. We conclude as in Step 2.

\medskip

\emph{Step 5.} 
In this step we prove the proposition when $R=D_4^\smone$. When $R\in\{D_5^\smone, D_7^\smone, E_6^\smone, E_7^\smone\}$, the proof is analogous.

We have $\Pi(R^\smperp)=\{\alpha_0,\dots,\alpha_4\}$. By \cite[Appendix B]{Sak01} we have
$$\textstyle \chi(\alpha_i) = \frac{a_i}{\lambda} \quad \text{for }i = 0, \dots , 4,$$
where $a_i$ are complex parameters and we set $\lambda \coloneqq a_0+\dots+a_4$. Set $ S\coloneqq \Spec\C[a_0,\dots,a_4,s,\lambda^{-1}]$, where we now view the parameters $a_i$ and $s$ as variables over $\C$. Then (a) and (b) follow as in Step 2.

For $\gamma$ as in Step 1, write $\gamma=m_0\alpha_0+\dots+m_4\alpha_4$, where $m_i\in\Z$. Then \eqref{eq:chigamma1} and the formulas above give
$$ \textstyle 0=\chi(\gamma)=\frac{1}{\lambda}\sum_{i=0}^4 m_ia_i. $$
Then it is clear that the set of solutions $(a_0,\dots,a_4,s)\in S$ of the equation above is the union of countably many linear subspaces. We conclude as in Step 2.

\medskip

\emph{Step 6.}
In this step we prove the proposition when $R=A_5^\smone$; when $R\in\{A_6^\smone, D_6^\smone\} $, the proof is analogous.

We have $\Pi(R^\smperp)=\{\alpha_0,\alpha_1,\alpha_2,\alpha_0',\alpha_1'\}$. By \cite[Appendix B]{Sak01} we have
\begin{align*}
e^{2\pi\sqrt{-1}\chi(\alpha_0)} &= b_0b_1/a_1a_2,  &e^{2\pi\sqrt{-1}\chi(\alpha_1)} = a_1, && e^{2\pi\sqrt{-1}\chi(\alpha_2)} &= a_2,\\
e^{2\pi\sqrt{-1}\chi(\alpha_0')} &= b_0,  &e^{2\pi\sqrt{-1}\chi(\alpha_1')} = b_1, && &
\end{align*}
where $a_i$ and $b_i$ are complex parameters. Set $S\coloneqq \Spec\C[a_1^{\pm1}, a_2^{\pm1}, b_0^{\pm1}, b_1^{\pm1}]$, where we now view the parameters $a_i$ and $b_i$ as variables over $\C$. Then (a) and (b) follow as in Step 2.

Write $\gamma=m_0\alpha_0+m_1\alpha_1+m_2\alpha_2+m_0'\alpha_0'+m_1'\alpha_1'$, where $m_i,m_j'\in\Z$. Then \eqref{eq:chigamma1} and the formulas above give
$$ 1=e^{2\pi\sqrt{-1}\chi(\gamma)}=a_1^{m_1-m_0}a_2^{m_2-m_0}b_0^{m_0+m_0'}b_1^{m_0+m_1'}. $$
Then it is clear that the set of solutions $(a_1,a_2,b_0,b_1)\in S$ of the equation above is the union of countably many analytic subsets. We conclude as in Step 2.
\end{proof}

	\bibliographystyle{amsalpha}
	\bibliography{biblio}

\newcommand{\etalchar}[1]{$^{#1}$}
\providecommand{\bysame}{\leavevmode\hbox to3em{\hrulefill}\thinspace}
\providecommand{\MR}{\relax\ifhmode\unskip\space\fi MR }
\providecommand{\MRhref}[2]{%
  \href{http://www.ams.org/mathscinet-getitem?mr=#1}{#2}
}
\providecommand{\href}[2]{#2}
\begin{thebibliography}{BCHM10}

\bibitem[AL11]{AL11}
M.~Artebani and A.~Laface, \emph{Cox rings of surfaces and the anticanonical
  {I}itaka dimension}, Adv. Math \textbf{226} (2011), no.~6, 5252--5267.

\bibitem[AM04]{AM04}
V.~Alexeev and S.~Mori, \emph{Bounding singular surfaces of general type},
  Algebra, arithmetic and geometry with applications ({W}est {L}afayette, {IN},
  2000), Springer, Berlin, 2004, pp.~143--174.

\bibitem[BCHM10]{BCHM}
C.~Birkar, P.~Cascini, C.~D. Hacon, and J.~M\textsuperscript{c}Kernan,
  \emph{Existence of minimal models for varieties of log general type}, J.
  Amer. Math. Soc. \textbf{23} (2010), no.~2, 405--468.

\bibitem[BH14]{BH14}
C.~Birkar and Z.~Hu, \emph{Polarized pairs, log minimal models, and {Z}ariski
  decompositions}, Nagoya Math. J. \textbf{215} (2014), 203--224.

\bibitem[Bor91]{Bor91}
C.~Borcea, \emph{On desingularized {H}orrocks-{M}umford quintics}, J. reine
  angew. Math. \textbf{421} (1991), 23--41.

\bibitem[Bou04]{Bou04}
S.~Boucksom, \emph{Divisorial {Z}ariski decompositions on compact complex
  manifolds}, Ann. Sci. \'Ecole Norm. Sup. (4) \textbf{37} (2004), no.~1,
  45--76.

\bibitem[BZ16]{BZ16}
C.~Birkar and D.-Q. Zhang, \emph{Effectivity of {I}itaka fibrations and
  pluricanonical systems of polarized pairs}, Publ. Math. Inst. Hautes \'Etudes
  Sci. \textbf{123} (2016), 283--331.

\bibitem[CD89]{CD89}
F.~R. Cossec and I.~V. Dolgachev, \emph{Enriques surfaces. {I}}, Progress in
  Mathematics, vol.~76, Birkh\"auser Boston, Inc., Boston, MA, 1989.

\bibitem[CD12]{CD12}
S.~Cantat and I.~Dolgachev, \emph{Rational surfaces with a large group of
  automorphisms}, {J. Amer. Math. Soc.} \textbf{25} (2012), no.~3, 863–905.

\bibitem[Deb01]{Deb01}
O.~Debarre, \emph{Higher-dimensional algebraic geometry}, Universitext,
  Springer-Verlag, New York, 2001.

\bibitem[Dol83]{Dol83}
I.~V. Dolgachev, \emph{Weyl groups and {C}remona transformations},
  Singularities, {P}art 1 ({A}rcata, {C}alif., 1981), Proc. Sympos. Pure Math.,
  vol.~40, Amer. Math. Soc., Providence, RI, 1983, pp.~283--294.

\bibitem[DPS94]{DPS94}
J.-P. Demailly, Th. Peternell, and M.~Schneider, \emph{Compact complex
  manifolds with numerically effective tangent bundles}, J. Algebraic Geom.
  \textbf{3} (1994), no.~2, 295--345.

\bibitem[D'U24]{DU24}
L.~D'Urso, \emph{On the {N}ef cones of blowups of the projective plane},
  arXiv:2412.15460\setbox0=\hbox{2024}.

\bibitem[Fri13]{Fri13}
R.~Friedman, \emph{On the ample cone of a rational surface with an
  anticanonical cycle}, Algebra Number Theory \textbf{7} (2013), no.~6, 1481 --
  1504.

\bibitem[Fri15]{Fri15}
\bysame, \emph{On the geometry of anticanonical pairs},
  arXiv:1502.02560\setbox0=\hbox{2015}.

\bibitem[Fuj79]{Fuj79}
T.~Fujita, \emph{On {Z}ariski problem}, Proc. Japan Acad. Ser. A Math. Sci.
  \textbf{55} (1979), no.~3, 106--110.

\bibitem[Fuj17]{Fuj17}
O.~Fujino, \emph{Foundations of the minimal model program}, MSJ Memoirs,
  vol.~35, Mathematical Society of Japan, Tokyo, 2017.

\bibitem[Gac25]{Gac25}
C.~Gachet, \emph{{The pseudoautomorphism group of $\mathbb{P}^3$ blown-up at
  $8$ very general points}}, arXiv:2509.12977\setbox0=\hbox{2025}.

\bibitem[GHK15]{GHK15}
M.~Gross, P.~Hacking, and S.~Keel, \emph{Moduli of surfaces with an
  anti-canonical cycle}, Compos. Math. \textbf{151} (2015), no.~2, 265–291.

\bibitem[Giz80]{Giz80}
M.~H. Gizatullin, \emph{Rational {$G$}-surfaces}, Izv. Akad. Nauk SSSR Ser.
  Mat. \textbf{44} (1980), no.~1, 110--144.

\bibitem[GKP16]{GKP16a}
D.~Greb, S.~Kebekus, and Th. Peternell, \emph{{\'E}tale fundamental groups of
  {K}awamata log terminal spaces, flat sheaves, and quotients of abelian
  varieties}, Duke Math. J. \textbf{165} (2016), no.~10, 1965--2004.

\bibitem[GLSW26]{GLSW26}
C.~Gachet, H.-Y. Lin, I.~Stenger, and L.~Wang, \emph{{The effective cone
  conjecture for Calabi--Yau pairs}}, J. reine angew. Math. (2026).

\bibitem[Har77]{Har77}
R.~Hartshorne, \emph{Algebraic geometry}, Graduate Texts in Mathematics,
  vol.~52, Springer-Verlag, New York, 1977.

\bibitem[HL20]{HanLiu20}
J.~Han and W.~Liu, \emph{On numerical nonvanishing for generalized log
  canonical pairs}, Doc. Math. \textbf{25} (2020), 93--123.

\bibitem[HL22]{HL22}
J.~Han and Z.~Li, \emph{Weak {Z}ariski decompositions and log terminal models
  for generalized pairs}, Math. Z. \textbf{302} (2022), no.~2, 707--741.

\bibitem[Kaw85]{Kaw85}
Y.~Kawamata, \emph{Pluricanonical systems on minimal algebraic varieties},
  Invent. Math. \textbf{79} (1985), 567--588.

\bibitem[Kaw97]{Kaw97}
\bysame, \emph{On the cone of divisors of {C}alabi-{Y}au fiber spaces},
  Internat.\ J.\ Math. \textbf{8} (1997), 665--687.

\bibitem[KM98]{KM98}
J.~Koll{\'a}r and S.~Mori, \emph{Birational geometry of algebraic varieties},
  Cambridge Tracts in Mathematics, vol. 134, Cambridge University Press,
  Cambridge, 1998.

\bibitem[Laz04]{Laz04}
R.~Lazarsfeld, \emph{Positivity in algebraic geometry. {I}, {II}}, Ergebnisse
  der Mathematik und ihrer Grenzgebiete, vol. 48, 49, Springer-Verlag, Berlin,
  2004.

\bibitem[Laz13]{Laz13}
V.~Lazi\'c, \emph{Around and beyond the canonical class}, Birational geometry,
  rational curves, and arithmetic, Simons Symp., Springer, Cham, 2013,
  pp.~171--203.

\bibitem[Laz26]{Laz26}
\bysame, \emph{A few remarks on effectivity and good minimal models}, Pure
  Appl. Math. Q. \textbf{22} (2026), no.~1, 221--234.

\bibitem[LH05]{LH05}
M.~Lahyane and B.~Harbourne, \emph{Irreducibility of -1-classes on
  anticanonical rational surfaces and finite generation of the effective
  monoid}, Pac. J. Math. \textbf{218} (2005), 101--114.

\bibitem[Li25]{Li25}
J.~Li, \emph{A cone conjecture for log {C}alabi-{Y}au surfaces}, Forum Math.
  Sigma \textbf{13} (2025), Paper No. e15.

\bibitem[LMP{\etalchar{+}}23]{LMPTX23}
V.~Lazi\'c, S.-i. Matsumura, Th. Peternell, N.~Tsakanikas, and Z.~Xie,
  \emph{The nonvanishing problem for varieties with nef anticanonical bundle},
  Doc. Math. \textbf{28} (2023), no.~6, 1393--1440.

\bibitem[Loo81]{Loo81}
E.~Looijenga, \emph{Rational surfaces with an anticanonical cycle}, Ann. of
  Math. \textbf{114} (1981), no.~2, 267--322.

\bibitem[Loo14]{Loo14}
\bysame, \emph{Discrete automorphism groups of convex cones of finite type},
  Compos. Math. \textbf{150} (2014), no.~11, 1939--1962.

\bibitem[LOP18]{LOP18}
V.~Lazi\'c, K.~Oguiso, and Th. Peternell, \emph{The {M}orrison-{K}awamata cone
  conjecture and abundance on {R}icci flat manifolds}, Uniformization,
  {R}iemann-{H}ilbert correspondence, {C}alabi-{Y}au manifolds \&
  {P}icard-{F}uchs equations, Adv. Lect. Math. (ALM), vol.~42, Int. Press,
  Somerville, MA, 2018, pp.~157--185.

\bibitem[LOP20]{LOP20}
\bysame, \emph{Nef line bundles on {C}alabi-{Y}au threefolds, {I}}, Int. Math.
  Res. Not. IMRN (2020), no.~19, 6070--6119.

\bibitem[LP18]{LP18a}
V.~Lazi\'c and Th. Peternell, \emph{Abundance for varieties with many
  differential forms}, \'Epijournal Geom. Alg\'ebrique \textbf{2} (2018),
  Article 1.

\bibitem[LP20a]{LP20a}
\bysame, \emph{{On Generalised Abundance, I}}, Publ. Res. Inst. Math. Sci.
  \textbf{56} (2020), no.~2, 353--389.

\bibitem[LP20b]{LP20b}
\bysame, \emph{{On Generalised Abundance, II}}, Peking Math. J. \textbf{3}
  (2020), no.~1, 1--46.

\bibitem[LSX26a]{LSX26}
V.~Lazi\'c, I.~Stenger, and Z.~Xie, \emph{{Generalised Cone Conjecture, II:
  Generic Sakai surfaces}}, arXiv:2608.27001\setbox0=\hbox{2026}.

\bibitem[LSX26b]{LSX26a}
\bysame, \emph{{Generalised Cone Conjecture, III: Non-generic Sakai surfaces}},
  arXiv:2610.10115\setbox0=\hbox{2026}.

\bibitem[LX25]{LX25}
V.~Lazi\'c and Z.~Xie, \emph{Nakayama-{Z}ariski decomposition and the
  termination of flips}, \'Epijournal G\'eom. Alg\'ebrique \textbf{9} (2025),
  Article 24.

\bibitem[LZ25]{LZ25}
Z.~Li and H.~Zhao, \emph{On the relative {M}orrison-{K}awamata cone
  conjecture}, Proc. Lond. Math. Soc. (3) \textbf{131} (2025), no.~5, Paper No.
  e70099.

\bibitem[Mat02]{Mat02}
K.~Matsuki, \emph{Introduction to the {M}ori {P}rogram}, Universitext,
  Springer--Verlag, New York, 2002.

\bibitem[Mor93]{Mor93}
D.~R. Morrison, \emph{Compactifications of moduli spaces inspired by mirror
  symmetry}, no. 218, 1993, Journ\'ees de G\'eom\'etrie Alg\'ebrique d'Orsay
  (Orsay, 1992), pp.~243--271.

\bibitem[Nag61]{Nag61}
M.~Nagata, \emph{On rational surfaces. {II}}, Mem. Coll. Sci. Univ. Kyoto Ser.
  A. Math. \textbf{33} (1960/61), 271--293.

\bibitem[Nak04]{Nak04}
N.~Nakayama, \emph{Zariski-decomposition and abundance}, MSJ Memoirs, vol.~14,
  Mathematical Society of Japan, Tokyo, 2004.

\bibitem[PS12]{PS12}
A.~Prendergast-Smith, \emph{The cone conjecture for some rational elliptic
  threefolds}, Math. Z. \textbf{272} (2012), no.~1-2, 589--605.

\bibitem[Sak84]{Sak84}
F.~Sakai, \emph{Anticanonical models of rational surfaces}, Math. Ann.
  \textbf{269} (1984), no.~3, 389--410.

\bibitem[Sak01]{Sak01}
H.~Sakai, \emph{Rational surfaces associated with affine root systems and
  geometry of the {P}ainlev\'e{} equations}, Comm. Math. Phys. \textbf{220}
  (2001), no.~1, 165--229.

\bibitem[SX23]{SX23}
I.~Stenger and Z.~Xie, \emph{{Cones of divisors on $\mathbb{P}^3$ blown up at
  eight very general points}}, to appear in Algebraic Geometry,
  arXiv:2303.12005\setbox0=\hbox{2023}.

\bibitem[Tot08]{Tot08}
B.~Totaro, \emph{Hilbert's 14th problem over finite fields and a conjecture on
  the cone of curves}, Compos. Math. \textbf{144} (2008), no.~5, 1176--1198.

\bibitem[Tot10]{Tot10}
\bysame, \emph{The cone conjecture for {C}alabi-{Y}au pairs in dimension two},
  Duke Math. J. \textbf{154} (2010), 241--263.

\bibitem[Xu24]{Xu24}
F.~Xu, \emph{On the cone conjecture for certain pairs of dimension at most 4},
  arXiv:2405.20899\setbox0=\hbox{2024}.

\bibitem[Zar62]{Zar62}
O.~Zariski, \emph{The theorem of {R}iemann-{R}och for high multiples of an
  effective divisor on an algebraic surface}, Ann. of Math. \textbf{76} (1962),
  no.~3, 560--615, with an appendix by {D}avid {M}umford.

\end{thebibliography}
	
\end{document}